\documentclass[graybox]{svmult}

\usepackage{mathptmx}       
\usepackage{helvet}         
\usepackage{courier}        
\usepackage{type1cm}        
\usepackage{makeidx}         
\usepackage{graphicx}        
\usepackage{multicol}        
\usepackage[bottom]{footmisc}
\usepackage{amssymb}
\usepackage{amsmath}
\usepackage{latexsym}

\newcommand{\Sum}{\displaystyle \sum}

\newcommand{\R}{{\mathbb R}}
\newcommand{\N}{{\mathbb N}}

\newcommand{\Z}{{\mathbb Z}}
\def\Supp{\,\hbox{\rm Supp}\,}
\def\d{\partial}

\def\ep{\varepsilon}
\def\Id{{\rm Id}\,}

\def\ddj{\dot\Delta_j}
\def\cF{{\mathcal F}}

\def\cC{{\mathcal C}}
\def\cS{{\mathcal S}}
\def\cM{{\mathcal M}}

\let\tilde=\widetilde

\makeindex             

\begin{document}  

\title*{New maximal regularity results for the heat equation in exterior  domains, and applications}
\titlerunning{Maximal regularity  for the heat equation in exterior domains}
\author{R. Danchin  and P. B. Mucha}
\institute{Rapha\"el Danchin \at Universit\'e Paris-Est, LAMA, UMR 8050 and Institut Universitaire de France,
 61 avenue du G\'en\'eral de Gaulle,
94010 Cr\'eteil Cedex, France, \email{raphael.danchin@u-pec.fr}
\and Piotr Bogus\l aw Mucha \at Instytut Matematyki Stosowanej i Mechaniki,
 Uniwersytet Wars\-zawski, 
ul. Banacha 2,  02-097 Warszawa, Poland. \email{p.mucha@mimuw.edu.pl}}
%
%
\maketitle

\abstract{This paper is dedicated to the proof  of new maximal regularity results involving Besov spaces
for the heat equation in the half-space or in bounded or exterior domains of $\R^n.$ 
We strive for \emph{time independent} a priori estimates in regularity spaces of type $L_1(0,T;X)$ where
$X$ stands for some \emph{homogeneous}  Besov space. In the case of bounded domains, 
the results that we get are similar to those of the whole space or of the half-space.
For exterior domains, we need to use mixed Besov norms in order to get 
a control on the low frequencies. Those estimates are crucial for proving 
global-in-time results for nonlinear heat equations in a critical functional framework.}

\section*{Introduction}

We are concerned with the proof of maximal regularity estimates for the heat equation
with Dirichlet boundary conditions, namely, 
\begin{equation}\label{eq:heat}
 \begin{array}{lcr}
 u_t-\nu \Delta u =f \qquad & \mbox{in} & (0,T)\times\Omega,\\
 u=0&\mbox{at} &(0,T)\times\d\Omega,\\
u = u_0 &  \mbox{on} & \Omega
\end{array}
\end{equation}
in various domains $\Omega$ of $\R^n$ ($n\geq2$). 
\medbreak
We are interested in $L_1$-in-time estimates for the solutions to \eqref{eq:heat} with a gain of two full spatial derivatives with respect to the data, that is
\begin{equation}\label{eq:maxregX}
\|u_t,\nu\nabla^2 u\|_{L_1(0,T;X)}\leq C\bigl(\|u_0\|_X+\|f\|_{L_1(0,T;X)}\bigr)
\end{equation}
with a constant  $C$ \emph{independent of $T.$} 
\medbreak
Such time independent estimates  are of importance not only for  the heat semi-group theory but
also in the applications. Typically, they are crucial for proving global existence and uniqueness 
statements for nonlinear heat equations with small data in a critical functional framework. 
Moreover, the fact that two full derivatives may be gained with respect to the source term
allows to consider not only the $-\Delta$  operator but also small perturbations of it.
In addition, we shall see below that it is possible  to choose $X$ in such a way that
the constructed solution $u$ is $L_1$-in-time with values in 
the set of Lipschitz functions. Hence, if the considered
nonlinear heat equation determines the velocity field of some fluid then 
this velocity field admits a unique Lipschitzian flow for all time.
The model may thus be reformulated equivalently in  Lagrangian variables
(see e.g.  our recent work \cite{DM-cpam} in the slightly different context of incompressible flows).
This is obviously of  interest to investigate free boundary problems. 
\bigbreak
Let us recall however that estimates such as \eqref{eq:maxregX} are  false  if $X$ is any 
reflexive Banach space, hence in particular if $X$ is a Lebesgue or Sobolev space (see e.g. \cite{DHP}). 
On the other hand, it is well known that \eqref{eq:maxregX}  holds true
in the whole space $\R^n$  if $X$ is a homogeneous Besov space \emph{with third index $1.$}
Let us be more specific. 
Let us fix some homogeneous Littlewood-Paley decomposition $(\ddj)_{j\in\Z}$
(see the definition in the next section) and denote
by  $(e^{\alpha\Delta})_{\alpha>0}$ the heat semi-group over $\R^n.$
Then it is well known  (see e.g. \cite{BCD}) that  there exist two constants $c$ and $C$
such that for all $j\in\Z$ and $\alpha\in\R^+$ one has
\begin{equation}\label{eq:ddj}
\|e^{\alpha\Delta}\ddj h\|_{L_p(\R^n)}\leq Ce^{-c\alpha2^{2j}}\|\ddj h\|_{L_p(\R^n)}.
\end{equation}
Hence if $u$ satisfies \eqref{eq:heat} then one may write
$$
\ddj u(t)=e^{\nu t\Delta}\ddj u_0+\int_0^te^{\nu(t-\tau)\Delta}\ddj f\,d\tau.
$$
Therefore, taking advantage of \eqref{eq:ddj}, we discover that
$$
\|\ddj u(t)\|_{L_p(\R^n)}\leq C\biggl(e^{-c\nu t2^{2j}}\|\ddj u_0\|_{L_p(\R^n)}+\int_0^t
e^{-c\nu(t-\tau)2^{2j}}\|\ddj f\|_{L_p(\R^n)}\,d\tau\biggr),
$$
whence
$$\displaylines{\quad\|\ddj u\|_{L_\infty(0,T;L_p(\R^n))}+
\nu2^{2j}\|\ddj u\|_{L_1(0,T;L_p(\R^n))}\hfill\cr\hfill\leq C\bigl( \|\ddj u_0\|_{L_p(\R^n)}
+\|\ddj f\|_{L_1(0,T;L_p(\R^n))}\bigr).}
$$
Multiplying the inequality by $2^{js}$ and summing up over $j,$
we thus eventually get  for some  absolute constant $C$ independent of  $\nu$ and $T,$
\begin{multline}\label{w2}
\|u\|_{L_\infty(0,T;\dot B^s_{p,1}(\R^n))}+
 \|u_t,\nu\nabla^2 u\|_{ L_1(0,T;\dot B^s_{p,1}(\R^n))}\\
\leq C(\|f\|_{ L_1(0,T;\dot B^s_{p,1}(\R^n))}+\|u_0\|_{\dot B^s_{p,1}(\R^n)}),
\end{multline}
where  the homogeneous Besov semi-norm that is used in the above
inequality is defined by
$$
\|u\|_{\dot B^s_{p,1}(\R^{n})}:=\sum_{j\in\Z} 2^{sj}\|\ddj u\|_{L_p(\R^{n})}.
$$
{}From this and the definition of  homogeneous Besov space $\dot B^s_{p,1}(\R^n)$
(see Section \ref{s:tools}), we easily deduce  the following classical result:
\begin{theorem}\label{th:whole}
 Let $p\in[1,\infty]$ and $s\in\R$. Let  $f \in L_1(0,T;\dot B^s_{p,1}(\R^n))$
and $u_0\in \dot B^s_{p,1}(\R^n)$.
Then \eqref{eq:heat}   with $\Omega=\R^n$ has a unique solution $u$ in 
$$
\cC([0,T);\dot B^s_{p,1}(\R^n))\quad\hbox{with}\quad 
\d_tu,\nabla^2u\in L_1(0,T;\dot B^s_{p,1}(\R^n))
$$
and \eqref{w2} is satisfied. 
\end{theorem}

The present paper is mainly devoted to generalizations of Theorem \ref{th:whole} to 
the half-space, bounded or exterior domains (that is the complement of bounded
simply connected domains), and applications to the global solvability 
of nonlinear heat equations. 
\medbreak
Proving maximal regularity estimates for general domains  essentially relies on Theorem \ref{th:whole} and  localization
techniques. More precisely, after localizing the equation thanks to a suitable resolution of unity, 
one has to estimate  ``interior  terms'' with support that do not intersect the boundary of 
$\Omega$ and ``boundary  terms'' the support of which meets $\d\Omega.$
In order to prove  \emph{interior estimates} that is bounds for  the interior terms, it suffices to resort to the theorem in the whole space, 
Theorem \ref{th:whole}, for those interior terms satisfy \eqref{eq:heat} (with suitable data)
once extended by zero onto the whole space. 
In contrast, the extension of the boundary terms by zero does not
satisfy \eqref{eq:heat} on $\R^n.$ However,  performing a change of variable
reduces their study to that of \eqref{eq:heat} on the half-space $\R^n_+.$
Therefore, proving maximal regularity estimates  in general domains 
mainly relies on such estimates on $\R^n$ and on $\R^n_+.$ As a matter of fact, we shall
see that the latter case stems from the former, by symmetrization, \emph{provided  
 $s$ is close enough to $0$}. 
In the case of a general domain, owing to change of variables and localization however, 
we shall obtain \eqref{w2} either \emph{up to low order terms} or 
with a \emph{time-dependent} constant $C.$ 
In a bounded domain, it turns out that Poincar\'e inequality (or equivalently the fact that 
the Dirichlet Laplacian operator has eigenvalues bounded away from $0$) allows
to prove an exponential decay which is sufficient to cancel out those lower order terms. 
In the case of an exterior domain, that decay turns out to be only algebraic
(at most $t^{-n/2}$ in dimension $n$). As a consequence, absorbing the lower
order terms will enforce us to use  mixed Besov norms and to assume 
that  $n\geq3.$

\bigbreak

The paper unfolds as follows. The basic tools for our analysis (Besov spaces on domains, product estimates,
embedding results)
are presented in the next section.
In Section \ref{s:maximal} we prove maximal regularity estimates 
similar to those of Theorem \ref{th:whole} first in 
the half-space and next in exterior or bounded domains.  
As an application, in the last section,  we establish global existence results for nonlinear heat equations with
small data  in a critical functional framework.


\section{Tools}\label{s:tools}

In this section, we introduce the main functional spaces and (harmonic analysis) tools 
that will be needed in this paper.

\subsection{Besov spaces on the whole space}

Throughout we fix a smooth nonincreasing radial function $\chi:\R^n\rightarrow[0,1]$
supported in $B(0,1)$ and such that $\chi\equiv1$ on $B(0,1/2),$ and set 
$\varphi(\xi):=\chi(\xi/2)-\chi(\xi).$
Note that this implies that  
$\varphi$ is valued in $[0,1],$
supported in $\{1/2\leq r\leq 2\}$ and that 
\begin{equation}\label{eq:phi}
\sum_{k\in\Z}\varphi(2^{-k}\xi)=1\quad\hbox{for all}\quad \xi\not=0.
\end{equation}
Then  we introduce the homogeneous Littlewood-Paley
decomposition  $(\dot\Delta_k)_{k\in\Z}$ over $\R^{n}$
by setting
$$\dot\Delta_ku:=\varphi(2^{-k}D)u=
{\mathcal F}^{-1}\bigl(\varphi(2^{-k}\cdot){\mathcal F}u\bigr).
$$
Above  $\cF$ stands for the Fourier transform on $\R^{n}.$
We also define  the low frequency cut-off $\dot S_k:=\chi(2^{-k}D).$

In order to  define Besov spaces on $\R^n,$ we first introduce the
following homogeneous semi-norms and nonhomogeneous Besov norms
(for all $s\in\R$ and $(p,r)\in[1,\infty]^2$): 
$$\begin{array}{lll}
\|u\|_{\dot B^s_{p,r}(\R^{n})}&:=&
\bigl\|2^{sk}\|\dot\Delta_ku\|_{L_p(\R^{n})}\bigr\|_{\ell_r(\Z)}\\[1ex]
\|u\|_{B^s_{p,r}(\R^{n})}&:=&
\bigl\|2^{sk}\|\dot\Delta_ku\|_{L_p(\R^{n})}\bigr\|_{\ell_r(\N)}
+\|\dot S_0u\|_{L_p(\R^n)}.\end{array}
$$
The nonhomogeneous Besov space $B^s_{p,r}(\R^n)$ is the set
of tempered distributions $u$ such that $\|u\|_{B^s_{p,r}(\R^n)}$ is finite. 
Following \cite{BCD}, we define the homogeneous Besov space  $\dot B^s_{p,r}(\R^{n})$ as 
 $$
 \dot B^s_{p,r}(\R^{n})=\left\{u\in{\cS}'_h(\R^{n})\,:\,
 \|u\|_{\dot B^s_{p,r}(\R^{n})}<\infty\right\},
 $$
 where ${\mathcal S}'_h(\R^{n})$ stands for
  the set of tempered distributions $u$ over $\R^{n}$
  such that for all smooth compactly supported function $\theta$
  over $\R^{n},$ we have
  $\lim_{\lambda\rightarrow+\infty}\theta(\lambda D)u=0$ in 
  $L_\infty(\R^{n}).$
  Note that  any distribution  $u\in{\cS}'_h(\R^n)$ satisfies
  $u=\sum_{k\in\Z}\dot\Delta_ku$ in $ {\cS}'_h(\R^n).$

\bigbreak
We shall make an extensive use of the following result (see the proof in e.g.  \cite{BCD,DM-ext}):
\begin{proposition}\label{p:product} Let $b^s_{p,r}$ denote
$\dot B^s_{p,r}(\R^n)$ or $B^s_{p,r}(\R^n).$ 
Then the following a priori estimates hold true:
\begin{itemize}
\item For any $s>0,$
$$
\|uv\|_{b^s_{p,r}}\lesssim \|u\|_{L^\infty}\|v\|_{b^s_{p,r}}+
\|v\|_{L^\infty}\|u\|_{b^s_{p,r}}.
$$
\item For any $s>0$ and $t>0,$
$$
\|uv\|_{b^{s}_{p,r}}\lesssim\|u\|_{L^\infty}\|v\|_{b^s_{p,r}}+\|v\|_{b^{-t}_{\infty,r}}\|u\|_{b^{s+t}_{p,\infty}}.
$$
\item For any $t>0$ and  $s>-n/p',$
$$
\|uv\|_{b^{s}_{p,r}}\lesssim\|u\|_{L^\infty}\|v\|_{b^s_{p,r}}
+\|u\|_{b^{n/p'}_{p',\infty}}\|v\|_{b^s_{p,r}}
+\|v\|_{b^{-t}_{\infty,r}}\|u\|_{b^{s+t}_{p,\infty}}.
$$
\item For any $q>1$ and $1-n/q\leq s\leq1,$
$$
\|uv\|_{b^0_{q,1}} \lesssim \|u\|_{b^s_{n,1}}\| v\|_{b^{1-s}_{q,1}}.
$$
\end{itemize}
\end{proposition}

As obviously  a smooth compactly supported  function belongs to any space 
$\dot B^{n/p}_{p,1}(\R^n)$ with $1\leq p\leq\infty,$ and to any Besov
space $B^\sigma_{p,1}(\R^n),$  we deduce from 
the previous proposition and embedding that (see the proof in \cite{DM-ext}):
\begin{corollary}\label{c:stabprod}
Let $\theta$ be in $\cC_c^\infty(\R^n).$ Then $u\mapsto \theta\,u$ is a continuous
mapping of $b^{s}_{p,r}(\R^n)$
\begin{itemize}
\item for any $s\in\R$ and $1\leq p,r\leq\infty,$ if $b^s_{p,r}(\R^n)=B^s_{p,r}(\R^n);$
\item for any $s\in\R$ and $1\leq p,r\leq\infty$ satisfying 
\begin{equation}\label{eq:B}
-n/p'<s<n/p\quad\!\!\bigl(-n/p<s\leq n/p  \hbox{ if } r=1,\quad\!\!
 -n/p'\leq s<n/p \hbox{ if } r=\infty\bigr)
\end{equation} 
 if $b^s_{p,r}(\R^n)=\dot B^s_{p,r}(\R^n).$
\end{itemize}
\end{corollary}

The following  proposition  allows us to compare the spaces 
 $B^s_{p,r}(\R^n)$ and $\dot B^s_{p,r}(\R^n)$ for compactly 
 supported functions\footnote{Without any support assumption, 
 it is obvious that if  $s$ is positive then we have $\|\cdot\|_{\dot B^s_{p,r}(\R^n)}
 \lesssim\|\cdot\|_{\dot B^s_{p,r}(\R^n)},$ and the opposite inequality  holds
true if $s$ is negative.} (see the proof in \cite{DM-ext}): 
\begin{proposition}\label{p:compbesov}
Let $1\leq p,r\leq\infty$ and  
 $s>-n/p'$ (or $s\geq -n/p'$ if $r=\infty$).  
 Then for any \underline{compactly supported} distribution $f$ we have
$$
f\in B^s_{p,r}(\R^n)\iff f\in \dot B^s_{p,r}(\R^n)
$$
and there exists a constant $C=C(s,p,r,n,K)$ (with $K=\Supp f$) such that
$$
C^{-1}\|f\|_{\dot B^s_{p,r}(\R^n)}\leq \|f\|_{B^s_{p,r}(\R^n)}
\leq C\|f\|_{\dot B^s_{p,r}(\R^n)}.
$$
\end{proposition}

 The following lemma will be useful for boundary estimates  (see the proof in \cite{DM-ext}): 
 \begin{lemma}\label{l:compo}
 Let $Z$ be a Lipschitz  diffeomorphism on $\R^n$ with $DZ$ and $DZ^{-1}$
 bounded,  $(p,r)\in[1,\infty]^2$
 and $s$ a real number satisfying \eqref{eq:B}.
 \begin{itemize}
 \item If in addition $s\in(-1,1)$ and $Z$ is  measure  preserving  then   the linear map
 $u\mapsto u\circ Z$ is continuous on $\dot B^s_{p,r}(\R^n).$
 \item In the general case, the map  $u\mapsto u\circ Z$ is continuous on $\dot B^s_{p,r}(\R^n)$ provided in addition
 $J_{Z^{-1}}\in \dot B^{n/p'}_{p',\infty}\cap L_\infty$ with $J_{Z^{-1}}:=|\det DZ^{-1}|.$
 \end{itemize}
  \end{lemma}


\subsection{Besov spaces on domains} 

We aim at extending  the definition of  homogeneous Besov spaces 
to  general domains. We proceed  by restriction  as follows\footnote{Nonhomogeneous
 Besov spaces on domains may be defined by the same token.}:
  \begin{definition}
  For $s\in\R$ and $1\leq p,q\leq\infty,$ we  
  define the homogeneous Besov space $\dot B^s_{p,q}(\Omega)$
over $\Omega$ as the restriction (in the distributional sense) of 
$\dot B^s_{p,q}(\R^{n})$ on $\Omega,$ that is
$$
\phi\in\dot B^s_{p,q}(\Omega) \iff
\phi=\psi_{|\Omega}\quad\hbox{for some}\quad
\psi\in \dot B^s_{p,q}(\R^{n}).
$$
We then set
$$
\|\phi\|_{\dot B^s_{p,q}(\Omega)}:=
\inf_{\psi_{|\Omega}=\phi}\|\psi\|_{\dot B^s_{p,q}(\R^{n})}.
$$
\end{definition}

The embedding, duality and interpolation properties of these Besov spaces
may be deduced from those  on $\R^n.$ As regards duality, we shall use repeatedly 
the following result:
\begin{proposition}\label{p:duality} 
If $-1+1/p<s<1/p$  (with $1\leq p,r<\infty$) then  
  the space  $\dot B^{-s}_{p',r'}(\Omega)$ may be identified with the dual space of 
 $\dot B^s_{p,r}(\Omega)$;  in the limit case $r=\infty$ then   $\dot B^{-s}_{p',1}(\Omega)$ 
 may be identified with the dual space of the completion  
 of $\cC_c^\infty(\R^n)$ for $\|\cdot\|_{\dot B^s_{p,\infty}(\Omega)}.$
 Furthermore, without any condition over $(s,p,r),$ we have
 $$
\biggl| \int_{\Omega} uv\,dx\biggr|\leq C\|u\|_{\dot B^s_{p,r}(\Omega)}
\|v\|_{\dot B^{-s}_{p',r'}(\Omega)}.
  $$
\end{proposition}
Similarly, some product laws for  Besov spaces on $\R^n$ may be extended to the domain case. 
We shall use the last inequality of Proposition \ref{p:product}
and also the following result that is proved in \cite{DM-ext}: 
\begin{proposition}\label{p:productdomain} Let $b^s_{p,r}(\Omega)$ denote $\dot B^s_{p,r}(\Omega)$
or $B^s_{p,r}(\Omega),$ and $\Omega$ be any domain of $\R^n.$
Then for any $p\in[1,\infty],$   $s$ such that  $-n/p'<s<n/p$  (or $-n/p'<s\leq n/p$ if $r=1,$
or $-n/p'\leq s<n/p$ if $r=\infty$),
 the following
inequality holds true:
$$
\|uv\|_{b^s_{p,r}(\Omega)}\leq C\|u\|_{b^{n/q}_{q,1}(\Omega)}\|v\|_{b^s_{p,r}(\Omega)}\quad\hbox{with }
\  q=\min(p,p').$$
\end{proposition}

A  very useful feature of Besov spaces is their interpolation properties.
 We refer to the books \cite{Be, Triebel} for the proof of the following statement.
\begin{proposition}\label{p:interpolation}
Let $b^s_{p,q}$ denote $B^s_{p,q}$ or $\dot B^s_{p,q}$; $s\in \R$, $p\in [1,\infty)$ and $q\in [1,\infty]$. The real interpolation of 
Besov spaces gives the following statement if $s_1\not=s_2$:
$$
\bigl( b^{s_1}_{p,q_1}(\Omega), b^{s_2}_{p,q_2}(\Omega)\bigr)_{\theta,q}= b^{s}_{p,q}(\Omega)
$$
with $s=\theta s_2+(1-\theta)s_1$, and $\frac 1p=\frac{\theta}{p_2}+\frac{1-\theta}{p_1}\cdotp$

Moreover, if $s_1\not=s_2,$ $t_1\not=t_2$ and if  $T: b^{s_1}_{p_1,q_1}(\Omega)+b^{s_2}_{p_2,q_2}(\Omega) \to b^{t_1}_{k_1,l_1}(\Omega) 
+ b^{t_2}_{k_2,l_2}(\Omega)$ is a linear map,  bounded from $b^{s_1}_{p_1,q_1}(\Omega)$ to $b^{t_1}_{k_1,l_1}(\Omega)$ and from 
$b^{s_2}_{p_2,q_2}(\Omega)$ 
to $b^{t_2}_{k_2,l_2}(\Omega)$ then
for any $\theta\in(0,1),$  the map $T$ is also bounded from
$b^s_{p,q}(\Omega)$ to $b^t_{k,q}(\Omega)$ with 
$$
s=\theta s_2+(1-\theta)s_1,\quad
t=\theta t_2+(1-\theta)t_1,\quad
\frac1p=\frac\theta{p_2}+\frac{1-\theta}{p_1},\quad
\frac1k=\frac\theta{k_2}+\frac{1-\theta}{k_1}\cdotp
$$
\end{proposition}

The  following
composition estimate  will be of constant use in the last section of this paper.
\begin{proposition}\label{p:compo} Let $f:\R^r\rightarrow\R$ be a $C^1$ function such that
$f(\vec 0)=0$ and, for some $m\geq1$ and $K\geq0,$
\begin{equation}\label{eq:growth}
|df(\vec u)|\leq K|\vec u|^{m-1}\quad\hbox{for all }\ \vec u\in\R^r.
\end{equation}
Then for all $s\in(0,1)$ and $1\leq p,q\leq\infty$ there exists a constant $C$ so that
\begin{equation}\label{eq:compo1}
\|f(\vec u)\|_{\dot B^s_{p,q}(\Omega)}\leq CK\|\vec u\|_{L_\infty(\Omega)}^{m-1}\|u\|_{\dot B^s_{p,q}(\Omega)}. 
\end{equation}
\end{proposition}
\begin{proof}
The proof relies on the characterization of the norm of $\dot B^s_{p,q}(\Omega)$ 
by finite differences, namely\footnote{Here we just consider the case  $q<\infty$ to shorten the presentation.}
\begin{equation}\label{eq:fd}
\|f(\vec u)\|_{\dot B^s_{p,q}(\Omega)}=\biggl(\int_{\Omega}\biggl(\int_{\Omega}
\frac{|f(\vec u(y))-f(\vec u(x))|^p}{|y-x|^{n+sp}}\,dy\biggr)^{\frac qp}\,dx\biggr)^{\frac1q}.
\end{equation}
Now the mean value formula implies that
$$
f(\vec u(y))-f(\vec u(x))=\biggl(\int_0^1df(\vec u(x)+t(\vec u(y)-\vec u(x)))\,dt\biggr)\cdot(\vec u(y)-\vec u(x)).
$$
Hence using  the growth assumption \eqref{eq:growth},
\begin{equation}\label{eq:fd1}
|f(\vec u(y))-f(\vec u(x))|\leq K\biggl(\int_0^1|\vec u(x)+t(\vec u(y)-\vec u(x))|^{m-1}\,dt\biggr)|\vec u(y)-\vec u(x)|.
\end{equation}
Therefore  we get
$$|f(\vec u(y))-f(\vec u(x))|\leq K\|\vec u\|_{L_\infty(\Omega)}^{m-1}|\vec u(y)-\vec u(x)|.$$
Inserting this latter inequality in \eqref{eq:fd}, we readily get \eqref{eq:compo1}.
\end{proof}

In \cite{DM-jfa,DM-ext}, we proved that:
\begin{proposition}\label{p:density}
Let $\Omega$ be the half-space, or a bounded or exterior domain with
$C^1$ boundary. 
 For all $1\leq p,q <\infty$, 
and $-1+1/p<s<1/p,$ we have
\begin{equation}
 B^s_{p,q}(\Omega)=\overline{\cC_c^\infty(\Omega)}^{\|\cdot\|_{B^s_{p,q}(\Omega)} }.
\end{equation}
\end{proposition}

\begin{remark}
In any $C^1$ domain $\Omega$ and for $0<s<n/p$ 
the  space  $\dot B^s_{p,q}(\Omega)$ embeds in $\dot B^0_{m,q}(\Omega)$ with $1/m=1/p-s/n.$
Therefore, if $q\leq\min(2,m),$ it also embeds in the Lebesgue space $L_m(\Omega).$
So finally if $s\in (0,\frac 1p)$ and $q\leq\min(2,m)$ with $m$ as above
then Proposition \ref{p:density} allows us to redefine the space $\dot B^s_{p,q}(\Omega)$ by
\begin{equation}
 \dot B^s_{p,q}(\Omega)=\overline{\cC_c^\infty(\Omega)}^{\|\cdot\|_{\dot B^s_{p,q}(\Omega)}}.
\end{equation}
\end{remark}

\begin{remark}\label{r:besov} 
In particular under the above hypotheses,  both classes  of Besov spaces
admit trivial extension by zero onto the whole space.
Combining with Proposition \ref{p:compbesov}, we deduce that
$$
 B^s_{p,q}(\Omega)=\dot  B^s_{p,q}(\Omega)\quad\hbox{if }\ 
 -1+1/p<s<1/p\quad\hbox{and}\quad\Omega \ \hbox{ is }\ bounded.
$$
Note also that, for obvious reasons, the above density result does not hold true if $q=\infty,$
for the strong topology. However, it holds for the weak $*$ topology.
\end{remark}


\section{A priori estimates for the heat equation}\label{s:maximal}

 This section is the core of the paper. Here we prove generalizations of Theorem \ref{th:whole}
 to more general domains. 
 First we consider the half-space case, then 
 we consider the exterior and bounded cases. 
We shall mainly focus on the unbounded case which is more tricky 
and just indicate at the end of this section what has to be changed
in the bounded domain case.

\subsection{The heat equation  in the half-space}

The purpose of this paragraph is to extend  Theorem \ref{th:whole}
to the half-space case $\R^n_+,$ namely
\begin{equation}\label{h1}
 \begin{array}{lcl}
 u_t-\nu \Delta u =f & \mbox{in } & (0,T)\times\R^n_+,\\[6pt]
u|_{x_n=0}=0 & \mbox{on } & (0,T)\times \d\R^{n}_+,\\[6pt]
u|_{t=0} = u_0 &  \mbox{on } & \R^n_+.
\end{array}
\end{equation}

\begin{theorem}\label{th:half}
Let $p\in[1,\infty)$ and  $s\in(-1+1/p,1/p).$
Assume that  $f$ belongs to $L_1(0,T;\dot B^s_{p,1}(\R^n_+))$
and that $u_0$ is  in $\dot B^s_{p,1}(\R^n_+).$
Then \eqref{h1} has a unique solution $u$ satisfying  
$$
u\in\cC([0,T);\dot B^s_{p,1}(\R^n_+)),\quad u_t,\nabla^2u\in L_1(0,T;\dot B^s_{p,1}(\R^n_+))
$$
and  the following estimate is valid:
\begin{multline}\label{h3}
\|u\|_{L_\infty(0,T;\dot B^s_{p,1}(\R^n_+))}+
  \|u_t,\nu\nabla^2 u\|_{ L_1(0,T;\dot B^s_{p,1}(\R^n_+))}\\\leq 
C(\|f\|_{ L_1(0,T;\dot B^s_{p,1}(\R^n_+))}+\|u_0\|_{\dot B^s_{p,1}(\R^n_+)}),
\end{multline}
where $C$ is an absolute constant with no dependence on $\nu$ and $T$.
\end{theorem}
\begin{proof} We argue by symmetrization.  Let $\tilde u_0$ and $\tilde f$ be the antisymmetric 
extensions over $\R^n$ to the data $u_0$ and $f.$ Then, given our assumptions over $s$ and Proposition \ref{p:density}, 
one may assert that $\tilde u_0\in\dot B^s_{p,1}(\R^n),$ $\tilde f\in L_1(0,T;\dot B^s_{p,1}(\R^n))$ and that, in addition
$$
\|\tilde u_0\|_{\dot B^s_{p,1}(\R^n)}\approx \| u_0\|_{\dot B^s_{p,1}(\R^n_+)}\quad\hbox{and}\quad
\|\tilde f\|_{L_1(0,T;\dot B^s_{p,1}(\R^n))}\approx \| f\|_{L_1(0,T;\dot B^s_{p,1}(\R^n_+))}.
$$
Let $\tilde u$ be the solution given by Theorem \ref{th:whole}. 
As this solution is unique in the corresponding functional framework, the symmetry properties of the data ensure
that $\tilde u$ is antisymmetric with respect to $\{x_n=0\}.$ As a consequence, it vanishes over $\{x_n=0\}.$
Hence the restriction $u$ of $\tilde u$ to the half-space satisfies \eqref{h1}. 
In addition,
\begin{itemize}
\item $\tilde u_t$ coincides with the antisymmetric extension of $u_t,$
\item  $\nabla^2_{x'}\tilde u$ coincides with the antisymmetric extension of $\nabla^2_{x'}u,$
\item $\nabla_{x'}\d_{x_n}\tilde u$ coincides with the symmetric extension of $\nabla_{x'}\d_{x_n}u,$
\item $\d^2_{x_n,x_n}\tilde u=(\Delta-\Delta_{x'})\tilde u$ hence coincides with 
$\tilde u_t-\tilde f-\Delta_{x'}\tilde u.$
\end{itemize}
Hence one may conclude that 
$$\displaylines{\quad
\|u\|_{L_\infty(0,T;\dot B^s_{p,1}(\R^n_+))}+
  \|u_t,\nu\nabla^2 u\|_{ L_1(0,T;\dot B^s_{p,1}(\R^n_+))}\hfill\cr\hfill\leq 
  \|\tilde u\|_{L_\infty(0,T;\dot B^s_{p,1}(\R^n))}+
  \|\tilde u_t,\nu\nabla^2\tilde u\|_{ L_1(0,T;\dot B^s_{p,1}(\R^n))}.\quad}
$$
This implies \eqref{h3}.
\end{proof}
\begin{remark} The case of  non-homogeneous boundary conditions  where $u$ equals some given $h$ at the boundary,
 reduces to the homogeneous case : it is only a matter of assuming that $h$ admits some
extension $\tilde h$ over $(0,T)\times\R^n_+$ so that 
$\tilde h_t-\nu\Delta\tilde h\in L_1(0,T;\dot B^s_{p,1}(\R^n_+)).$ 
\end{remark}


\subsection{The exterior domain case}

Here we extend Theorem \ref{th:whole} to the case where 
$\Omega$ is an exterior domain (that is the complement of a bounded simply connected
domain). Here is our main statement:

\begin{theorem}\label{th:heat-fin}
 Let  $\Omega$ be a $C^2$ exterior domain of $\R^n$ with $n\geq3.$
 Let $1<q\leq p<\infty$ with $q< n/2.$ Let  $-1 +1/p < s < 1/p$
 and $-1+1/q<s'<1/q-2/n.$ Let
$$u_0\in \dot B^s_{p,1}\cap \dot B^{s'}_{q,1}(\Omega)\quad\hbox{and}\quad  f\in L_1(0,T;\dot B^s_{p,1}\cap \dot B^{s'}_{q,1}(\Omega)).
$$
Then there exists a unique solution $u$ to \eqref{eq:heat} such that
$$
u\in\cC([0,T];\dot B^{s}_{p,1}\cap\dot B^{s'}_{q,1}(\Omega)), \qquad 
u_t,\nabla^2u \in L_1(0,T;\dot B^s_{p,1}\cap B^{s'}_{q,1}(\Omega))
$$
and the following inequality is satisfied:
 \begin{multline}\label{fin1}
 \|u\|_{L_\infty(0,T;\dot B^s_{p,1}\cap \dot B^{s'}_{q,1}(\Omega))}\! 
+\! \| u_t,\nu\nabla^2 u\|_{L_1(0,T;\dot B^s_{p,1}\cap \dot B^{s'}_{q,1}(\Omega))} 
\\\leq C\big( \|u_0\|_{\dot B^s_{p,1}\cap \dot B^{s'}_{q,1}(\Omega)}+
\|f\|_{L_1(0,T;\dot B^s_{p,1}\cap \dot B^{s'}_{q,1}(\Omega))} \big),\qquad 
\end{multline}
where the constant $C$ is independent of $T$ and $\nu.$
\end{theorem}
Proving this theorem  relies on the 
following statement (that is of independent interest and holds in any 
dimension $n\geq2$), and on lower order estimates (see Lemma \ref{l:vitesse} below)
which will enable us to remove the time dependency.
\begin{theorem}\label{th:ext} 
Let $\Omega$ be a $C^2$ exterior domain of  $\R^n$ with $n\geq2.$
 Let $1<p<\infty$, $-1+1/p<s<1/p$, $f \in L_1(0,T;\dot B^s_{p,1}(\Omega))$,
 and $u_0\in \dot B^s_{p,1}(\Omega)$. 
 Then   equation \eqref{eq:heat} has a unique solution $u$  such that
$$
u\in\cC([0,T];\dot B^s_{p,1}(\Omega)), \qquad 
\d_tu,\nabla^2u \in L_1(0,T;\dot B^s_{p,1}(\Omega))
$$
and the following estimate is valid:
\begin{multline}\label{estdep}
\|u\|_{L_\infty(0,T;\dot B^s_{p,1}(\Omega))}\!+\!
\|u_t,\nu\nabla^2 u\|_{L_1(0,T;\dot B^s_{p,1}(\Omega))}
\\\leq Ce^{CT\nu}\big(\|u_0\|_{\dot B^{s}_{p,1}(\Omega)}+\|f\|_{L_1(0,T;\dot B^s_{p,1}(\Omega))}\big),
\end{multline}
where the constant $C$ depends only  on $s,$ $p,$ and $\Omega.$ 
\smallbreak
Additionally  if  $K$ is a compact subset of $\Omega$ such that 
${\rm dist}(\partial \Omega,\Omega\setminus K)>0,$ there holds
\begin{multline}\label{estaux}
\|u\|_{L_\infty(0,T;\dot B^s_{p,1}(\Omega))}\!+\!
 \|u_t,\nu\nabla^2 u\|_{L_1(0,T;\dot B^s_{p,1}(\Omega))}\qquad\qquad\qquad\\[9pt]\leq C\bigl(\|u_0\|_{\dot B^{s}_{p,1}(\Omega)}
+\|f\|_{L_1(0,T;\dot B^s_{p,1}(\Omega))}+\nu\|u\|_{L_1(0,T;\dot B^s_{p,1}(K))}\bigr),\qquad
\end{multline}
where $C$ is as above.
\end{theorem}
\begin{proof}
We suppose  that we have a smooth enough 
solution and focus on the proof of the estimates.
We shall do it in three steps: first we prove interior estimates, 
next boundary estimates and finally global estimates after summation.

Note that performing the following  change of unknown and data:
$$
u_{new}(t,x)=\nu u_{old}(\nu^{-1}t,x),\quad
u_{0,new}(x)=\nu u_{0,old}(x),\quad
f_{new}(t,x)=f_{old}(\nu^{-1}t,x)
 $$
reduces the study to the case $\nu=1.$ 
So we shall make this assumption in all that follows.
\smallbreak
Throughout  we fix some covering $(B(x^\ell,\lambda))_{1\leq\ell\leq L}$ of $K$
by balls of radius $\lambda$ and take some neighborhood $\Omega^0\subset\Omega$ of $\R^n\setminus K$
such that $d(\Omega^0,\d\Omega)>0.$ 
We assume in addition that the first $M$ balls do not intersect $K$ while the last $L-M$ balls
are centered at some point of $\d\Omega.$

Let  $\eta^0:\R^n\rightarrow[0,1]$ be  a smooth  function  supported in $\Omega^0$ and 
with  value $1$ on a neighborhood of $\Omega\setminus K.$  
Then we  consider
 a subordinate partition of unity $(\eta^\ell)_{1\leq\ell\leq L}$ such that:
 \begin{enumerate}
\item  $\sum_{0\leq l\leq L}\eta^\ell=1\quad\hbox{on }\ \Omega$;
\item $\|\nabla^k\eta^\ell\|_{L_\infty(\R^n)}\leq C_k\lambda^{-k}$ for $k\in\N$ and $1\leq\ell\leq L$;
\item $\Supp\eta^\ell\subset B(x^\ell,\lambda).$
\end{enumerate}
We also introduce another smooth function $\tilde\eta^0$ supported
in $K$ and with value $1$ on $\Supp\nabla\eta^0$
and smooth functions  $\tilde\eta^1,\cdots,\tilde\eta^L$ 
with compact support in $\Omega^\ell$ and such that $\tilde\eta^\ell\equiv1$ on
$\Supp\eta^\ell.$ 

Note that for $\ell\in\{1,\cdots,L\},$  the bounds for the derivatives of $\eta^\ell$ 
together with the fact that $\bigl|\Supp\nabla\eta^\ell\bigr|\approx\lambda^n$ 
and Proposition \ref{p:interpolation} implies that for $k=0,1$ and any $q\in[1,\infty],$ we have
\begin{equation}\label{eq:lambda}
\|\nabla\eta^\ell\|_{\dot B^{k+n/q}_{q,1}(\R^n)}\lesssim\lambda^{-1-k}.
\end{equation}
The same holds for the functions $\tilde\eta^\ell.$
Throughout, we set    $U^\ell:=u\eta^\ell$.


\subsubsection*{First step: the interior estimate}

The vector-field  $U^0$ satisfies
the following modification of \eqref{eq:heat}:
\begin{equation}\label{i1}
 \begin{array}{lcr}
 U_t^0-\Delta U^0 =\eta^0f-2\nabla \eta^0\cdot\nabla u -  u \Delta \eta^0 
& \mbox{in} & (0,T)\times\R^n,\\[1ex]
U^0|_{t=0}=u_0  \eta^0 & \mbox{on} & \R^n.
\end{array}
\end{equation}
Theorem \ref{th:whole}  thus yields the following estimate:
$$
\displaylines{
\|U^0\|_{L_\infty(0,T;\dot B^s_{p,1}(\R^n))}+
 \|U_t^0,\nabla^2 U^0\|_{L_1(0,T;\dot B^s_{p,1}(\R^n))}
\lesssim\|\eta^0f\|_{L_1(0,T;\dot B^s_{p,1}(\R^n))} \hfill\cr\hfill+\|\nabla \eta^0\cdot\nabla u\|_{L_1(0,T;\dot B^s_{p,1}(\R^n))} +
\|u\Delta\eta^0\|_{L_1(0,T;\dot B^s_{p,1}(\R^n))} +
\|\eta^0 u_0\|_{\dot B^s_{p,1}(\R^n)}.}$$
Let us emphasize that as $\nabla \eta^0\cdot\nabla u$ and $u \Delta \eta^0$ are compactly supported, 
we may replace the homogeneous norms by non-homogeneous ones in the first two terms. 
As a consequence, because the function $\nabla\eta^0$ is in $\cC_c^\infty(\R^n)$
and $\tilde \eta^0\equiv1$ on $\Supp\nabla \eta^0,$ 
Corollary \ref{c:stabprod} ensures that 
\begin{multline}\label{i4a}
\|U^0\|_{L_\infty(0,T;\dot B^s_{p,1}(\R^n))}+
 \|U_t^0,\nabla^2 U^0\|_{L_1(0,T;\dot B^s_{p,1}(\R^n))}\\\lesssim
  \|\eta^0u_0\|_{\dot B^{s}_{p,1}(\R^n))}+
\|\eta^0f\|_{L_1(0,T;\dot B^s_{p,1}(\R^n))}+ \|\tilde\eta^0u\|_{L_1(0,T;B^{s+1}_{p,1}(\R^n))}.
\end{multline}

Now, by interpolation,
\begin{equation}\label{eq:interpo0}
\|\tilde\eta^0u\|_{B^{1+s}_{p,1}(\Omega)}\leq  C
\|\tilde\eta^0u\|_{B^{2+s}_{p,1}(\Omega)}^{\frac12}
\|\tilde\eta^0u\|_{B^s_{p,1}(\Omega)}^{\frac12}.
\end{equation}
As $\Supp\tilde\eta^0\subset K$ and as homogeneous and nonhomogeneous norms are equivalent on $K,$
 one may thus conclude that 
\begin{multline}\label{i5}
 \|U^0\|_{L_\infty(0,T;\dot B^s_{p,1}(\R^n))}+
 \|U^0_t,\nabla^2 U^0\|_{L_1(0,T;\dot B^s_{p,1}(\R^n))}\lesssim
 \|f\|_{L_1(0,T;\dot B^s_{p,1}(\Omega))}\\
 +T^{1/2}\|u\|_{L_1(0,T;\dot B^{2+s}_{p,1}(K))\cap L_\infty(0,T;\dot B^s_{p,1}(K))}+\|u_0 \|_{\dot B^s_{p,1}(\Omega)}.
\end{multline}
Note that starting from \eqref{eq:interpo0} and using Young's inequality also yields  for all $\ep>0$:  
\begin{multline}\label{i8}
 \|U^0\|_{L_\infty(0,T;\dot B^s_{p,1}(\R^n))}+
 \|U^0_t,\nabla^2 U^0\|_{L_1(0,T;\dot B^s_{p,1}(\R^n))}
 \leq C\bigl(\|u_0\|_{\dot B^s_{p,1}(\Omega)}\\+\|f\|_{L_1(0,T;\dot B^s_{p,1}(\Omega))}\bigr)
+\varepsilon \|u\|_{ L_1(0,T;\dot B^{2+s}_{p,1}(K))}
+c(\varepsilon) \|u\|_{ L_1(0,T;\dot B^s_{p,1}(K))}.
\end{multline}
The terms $U^\ell$ with $1\leq\ell\leq M$ may be bounded 
exactly along the same lines because their support do not meet $\d\Omega,$ hence
their extension by $0$ over $\R^n$ satisfies
 $$
 \begin{array}{lcr}
 U_t^\ell-\Delta U^\ell =f^\ell & \mbox{in} & (0,T)\times\R^n,\\
U^\ell|_{t=0}=u_0  \eta^\ell & \mbox{on} & \R^n
\end{array}
$$
with 
\begin{equation}\label{eq:fl0}
f^\ell:=-2 \nabla \eta^\ell\cdot\nabla u -  u \Delta \eta^\ell+\eta^\ell f.
\end{equation}
Arguing as above and taking advantage of the fact that the functions $\eta^\ell$
are localized in balls of radius $\lambda$  
(that is we use \eqref{eq:lambda}),  we now get
\begin{multline}\label{eq:fl}
\|f^\ell\|_{L_1(0,T;\dot B^s_{p,1}(\Omega))}\lesssim \|\eta^\ell f\|_{L_1(0,T;\dot B^s_{p,1}(\Omega))}
\\+\lambda^{-2}\|\tilde\eta^\ell u\|_{L_1(0,T;\dot B^s_{p,1}(\Omega))}
+\lambda^{-1}\|\tilde\eta^\ell\nabla u\|_{L_1(0,T;\dot B^s_{p,1}(\Omega))}.
\end{multline}
Using again \eqref{eq:interpo0} (with $\tilde\eta^\ell$ instead of $\tilde\eta^0$), we get
\begin{multline}\label{i5bis}
 \|U^\ell\|_{L_\infty(0,T;\dot B^s_{p,1}(\R^n))}+
 \|U^\ell_t,\nabla^2 U^\ell\|_{L_1(0,T;\dot B^s_{p,1}(\R^n))}\lesssim
 \|\eta^\ell f\|_{L_1(0,T;\dot B^s_{p,1}(\Omega))}\\
 +\bigl(\lambda^{-1}T^{1/2}+\lambda^{-2}T\bigr)
 \|u\|_{L_1(0,T;\dot B^{2+s}_{p,1}(K))\cap L_\infty(0,T;\dot B^s_{p,1}(K))}
 +\|u_0\eta^\ell \|_{\dot B^s_{p,1}(\Omega)},
\end{multline}  
\begin{multline}\label{i8bis}
 \|U^\ell\|_{L_\infty(0,T;\dot B^s_{p,1}(\R^n))}+
 \|U^\ell_t,\nabla^2 U^\ell\|_{L_1(0,T;\dot B^s_{p,1}(\R^n))}
 \leq C\bigl(\|u_0\eta^\ell\|_{\dot B^s_{p,1}(\Omega)}\\+\|\eta^\ell f\|_{L_1(0,T;\dot B^s_{p,1}(\Omega))}\bigr)
+\lambda^{-1}\|u\|_{ L_1(0,T;\dot B^{2+s}_{p,1}(K))}^{1/2}\|u\|_{ L_1(0,T;\dot B^s_{p,1}(K))}^{1/2}.
\end{multline}


\subsubsection*{Second step: the boundary estimate}

 We now consider an index  $\ell \in\{L+1,\cdots,M\}$ so that  $B(x^\ell,\lambda)$
 is  centered at a point of $\d\Omega.$
The localization leads to the following problem:
\begin{equation}\label{i1-bd}
 \begin{array}{lcr}
 U_t^\ell- \Delta U^\ell =f^\ell & \mbox{in} & (0,T)\times \Omega,\\
U^\ell=0 & \mbox{on}& (0,T)\times \d \Omega,\\
U_t^\ell|_{t=0}=u_0 \eta^\ell & \mbox{on} & \Omega,
\end{array}
\end{equation}
with $f^\ell$ defined by \eqref{eq:fl0}, hence satisfying \eqref{eq:fl}. 
\medbreak
 Let us now make a change of variables so as to recast \eqref{i1-bd} in the  half-space.
As $\d\Omega$ is $C^2,$ 
 if $\lambda$ has been chosen small enough then for fixed $\ell$ we are able to find
a map $Z_\ell$ so that
\begin{enumerate}
\item[i)] $Z_\ell$ is a $C^2$ diffeomorphism from $B(x^\ell,\lambda)$ 
to $Z_\ell(B(x^\ell,\lambda))$;
\item[ii)] $Z_\ell(x^\ell)=0$ and   $D_xZ(x^\ell)=\Id$;
\item[iii)] $Z_\ell(\Omega\cap B(x^\ell,\lambda))\subset \R^n_+$;
\item[iv)]  $Z_\ell(\d\Omega\cap B(x^\ell,\lambda))= \d\R_+^n\cap Z_\ell(B(x^\ell,\lambda)).$
\end{enumerate}
Setting $\nabla_x Z_\ell=\Id+A_\ell$ then one may assume in addition that 
 there exist constants $C_j$ depending only on $\Omega$ 
 and on $j\in\{0,1\}$ such that
 \begin{equation}\label{eq:b0}
 \|D^j A_\ell\|_{L_\infty(B(x^\ell,\lambda))}\leq  C_j,
 \end{equation}
 a property  which implies (by the mean value formula) that 
 \begin{equation}\label{eq:b1}
 \|A_\ell\|_{L_\infty(B(x^\ell,\lambda))}\leq C_1\lambda,
 \end{equation}
 hence  by interpolation between the spaces $L_q(B(x^\ell,\lambda))$ and $W^{r-1}_q(B(x^\ell,\lambda)),$
  \begin{equation}\label{eq:b00}
 \|A_\ell\|_{B^{\frac nq}_{q,1}(B(x^\ell,\lambda))}
 \leq C\lambda\quad\hbox{for all }\ 1\leq q<\infty
 \ \hbox{ such that }\  n/q<r-1.
 \end{equation}

 Let $V^\ell:=Z_\ell^*U^\ell:=U^\ell\circ Z_\ell^{-1}.$ The system satisfied by
$V^\ell$ reads 
\begin{equation}\label{b1}
 \begin{array}{lcl}
 V_t^\ell- \Delta_z V^\ell =F^\ell
& \mbox{in} & (0,T)\times\R^n_+,\\
V^\ell|_{z_n=0}=0 & \mbox{on} & (0,T)\times\d\R^{n}_+,\\
V^\ell|_{t=0}=Z^*_\ell(U^\ell|_{t=0})  & \mbox{on} & \d\R^{n}_+,\\
\end{array}
\end{equation}
with 
$$
F^\ell:= Z_\ell^*f^\ell+(\Delta_x-\Delta_z)V^\ell.
$$
According to Theorem \ref{th:half}, we thus get 
$$\displaylines{
\|V^\ell\|_{L_\infty(0,T;\dot B^s_{p,1}(\R^n_+))}+
 \|V_t^\ell,\nabla^2_zV^\ell\|_{L_1(0,T;\dot B^s_{p,1}(\R^n_+))}\hfill\cr\hfill\lesssim
\|Z_\ell^*f^\ell \|\|_{L_1(0,T;\dot B^s_{p,1}(\R^n_+))}
+(\Delta_x-\Delta_z)V^\ell)\|_{L_1(0,T;\dot B^s_{p,1}(\R^n_+))}
+\|Z_\ell^*(U^\ell|_{t=0})\|_{\dot B^s_{p,1}(\R^n_+)}.}
$$
 Note that the first and last terms in the right-hand side may be dealt with thanks
 to Lemma \ref{l:compo}: we have
 $$\begin{array}{l}
 \|Z_\ell^*f^\ell\|_{L_1(0,T;\dot B^s_{p,1}(\R^n_+))}
 \lesssim \|f^\ell\|_{L_1(0,T;\dot B^s_{p,1}(\Omega))}\\[1ex]
 \|Z_\ell^*(U^\ell|_{t=0})\|_{\dot B^s_{p,1}(\R^n_+)}\lesssim
 \|U^\ell|_{t=0}\|_{\dot B^s_{p,1}(\Omega)}.\end{array}
$$
Compared to the first step, the only definitely  new term is   $(\Delta_x-\Delta_z)V^\ell.$
Explicit computations (see e.g. \cite{DM-ext}) show that
$(\Delta_z-\Delta_x)V^\ell$ is a linear combination of components  of  
$\nabla_z^2 A_\ell \otimes V^\ell$ and $\nabla_z A_\ell \otimes\nabla_z V^\ell.$
Therefore
$$\displaylines{\quad
\|(\Delta_x-\Delta_z)V^\ell)\|_{L_1(0,T;\dot B^s_{p,1}(\R^n_+))}
\lesssim \|A_\ell\otimes\nabla_z^2 V^\ell\|_{L_1(0,T;\dot B^{s}_{p,1}(\R^n_+))}
\hfill\cr\hfill+\|\nabla_z A_\ell\otimes\nabla_zV^\ell\|_{L_1(0,T;\dot B^s_{p,1}(\R^n_+))}.\quad}
$$

Now, according to Proposition \ref{p:productdomain} and owing  
to the support properties of the terms involved in the inequalities,, we have 
$$
\|A_\ell\otimes\nabla_z^2 V^\ell\|_{L_1(0,T;\dot B^{s}_{p,1}(\R^n_+))}
\lesssim  \|A_\ell\|_{\dot B^{\frac nq}_{q,1}(B(x^\ell,\lambda))}\|\nabla^2_z V^\ell\|_{\dot B^s_{p,1}(\R^n_+)}\quad\!\!\!\hbox{with}\quad\!\!\!
q=\min(p,p').
$$
Therefore we have, thanks to \eqref{eq:b1} and to \eqref{eq:b00}, 
$$
\|A_\ell\otimes\nabla_z^2 V^\ell\|_{L_1(0,T;\dot B^{s}_{p,1}(\R^n_+))}
\lesssim \lambda \|\nabla^2_z V^\ell\|_{\dot B^s_{p,1}(\R^n_+)}.
$$
Similarly, we have 
$$
\|\nabla_z A_\ell\otimes\nabla_zV^\ell\|_{L_1(0,T;\dot B^s_{p,1}(\R^n_+))}
\lesssim \|\nabla_zV^\ell\|_{L_1(0,T;\dot B^s_{p,1}(\R^n_+))}.
$$
Therefore
$$
\|(\Delta_x-\Delta_z)V^\ell\|_{L_1(0,T;\dot B^s_{p,1}(\R^n_+))}\lesssim \lambda\|\nabla^2_z V^\ell\|_{L_1(0,T;\dot B^s_{p,1}(\R^n_+))}
+\|\nabla_z V^\ell\|_{L_1(0,T;\dot B^s_{p,1}(\R^n_+))}.
 $$
 Putting together the above inequalities and remembering of \eqref{eq:fl} and Lemma \ref{l:compo}, we finally get, taking
 $\lambda$ small enough
$$\displaylines{
\|V^\ell\|_{L_\infty(0,T;\dot B^s_{p,1}(\R^n_+))}+
 \|V_t^\ell,\nabla^2_z V^\ell\|_{L_1(0,T;\dot B^s_{p,1}(\R^n_+))} \hfill\cr\hfill
 \lesssim \|U^\ell|_{t=0}\|_{\dot B^s_{p,1}(\Omega)} 
 +\|\eta^\ell f\|_{L_1(0,T;\dot B^s_{p,1}(\Omega))} \hfill\cr\hfill
+\lambda^{-2}\|\tilde\eta^\ell u\|_{L_1(0,T;\dot B^s_{p,1}(\Omega))}
+\lambda^{-1}\|\tilde\eta^\ell\nabla u\|_{L_1(0,T;\dot B^s_{p,1}(\Omega))}
+\|\nabla V^\ell\|_{L_1(0,T;\dot B^{s}_{p,1}(\R^n_+))}.}
$$
 
By interpolation, we have $$
 \|\nabla V^\ell\|_{L_1(0,T;\dot B^s_{p,1}(\R^n_+))} \lesssim
  \|\nabla^2 V^\ell\|^{1/2}_{L_1(0,T;\dot B^s_{p,1}(\R^n_+))}
  \|V^\ell\|^{1/2}_{L_1(0,T;\dot B^s_{p,1}(\R^n_+))}.
  $$
  Therefore using Young's inequality enables us to reduce the above inequality to 
  $$\displaylines{
\|V^\ell\|_{L_\infty(0,T;\dot B^s_{p,1}(\R^n_+))}+
 \|V_t^\ell,\nabla^2 V^\ell\|_{L_1(0,T;\dot B^s_{p,1}(\R^n_+))}\hfill\cr\hfill
 \lesssim \|U^\ell|_{t=0}\|_{\dot B^s_{p,1}(\Omega)}  
 +\|\eta^\ell f\|_{L_1(0,T;\dot B^s_{p,1}(\Omega))}+\|V^\ell\|_{L_1(0,T;\dot B^{s}_{p,1}(\R^n_+))}
\hfill\cr\hfill+\lambda^{-2}\|\tilde\eta^\ell u\|_{L_1(0,T;\dot B^s_{p,1}(\Omega))}
+\lambda^{-1}\|\tilde\eta^\ell\nabla u\|_{L_1(0,T;\dot B^s_{p,1}(\Omega))}.}
$$

In order to handle the last term, there are two ways of proceeding depending
on whether we want a time dependent constant or not. 
The first possibility is to write that, by interpolation and H\"older's inequality, 
$$
\begin{array}{lll}
\|\tilde\eta^\ell \nabla u\|_{L_1(0,T;\dot B^{s}_{p,1}(\Omega))}
&\leq& T^{1/2}\|u\|_{L_1(0,T;\dot B^{s+2}_{p,1}(K))\cap L_\infty(0,T;\dot B^{s}_{p,1}(K))}.
\end{array}
$$
This yields 
  \begin{multline}\label{b5}
\|V^\ell\|_{L_\infty(0,T;\dot B^s_{p,1}(\R^n_+))}+
 \|V_t^\ell,\nabla^2 V^\ell\|_{L_1(0,T;\dot B^s_{p,1}(\R^n_+))} \\
 \lesssim \|\eta^\ell u_0\|_{\dot B^s_{p,1}(\Omega)}
 +\|\eta^\ell f\|_{L_1(0,T;\dot B^s_{p,1}(\Omega))}+T\|V^\ell\|_{L_\infty(0,T;\dot B^s_{p,1}(\R^n_+))}\\
+\bigl(\lambda^{-1} T^{1/2}+\lambda^{-2}T\bigr)
 \|u\|_{L_1(0,T;\dot B^{s+2}_{p,1}(K))\cap L_\infty(0,T;\dot B^{s}_{p,1}(K))}.
\end{multline}
The second possibility is to write that
$$
\|\tilde\eta^\ell \nabla u\|_{L_1(0,T;\dot B^{s}_{p,1}(\Omega))}\leq \|u\|_{L_1(0,T;\dot B^{s+2}_{p,1}(K))}^{\frac12} 
\|u\|_{L_1(0,T;\dot B^{s}_{p,1}(K))}^{\frac12}.
$$
We eventually get
  \begin{multline}\label{b8}
\|V^\ell\|_{L_\infty(0,T;\dot B^s_{p,1}(\R^n_+))}+
 \|V_t^\ell,\nabla^2 V^\ell\|_{L_1(0,T;\dot B^s_{p,1}(\R^n_+))}  \lesssim \|\eta^\ell u_0\|_{\dot B^s_{p,1}(\Omega)}
 \\+\|\eta^\ell f\|_{L_1(0,T;\dot B^s_{p,1}(\Omega))}+ \lambda^{-1}\|u\|_{L_1(0,T;\dot B^{2+s}_{p,1}(K))}^{1/2}\|u\|_{ L_1(0,T;\dot B^s_{p,1}(K))}^{1/2}\\
 +\lambda^{-2}\|u\|_{L_1(0,T;\dot B^s_{p,1}(K))}
+\|\nabla V^\ell\|_{L_1(0,T;\dot B^s_{p,1}(\R^n_+))}.
\end{multline}


\subsubsection*{Third step: global a priori estimates}
Now, in view of Lemma \ref{l:compo}, we may write
$$\begin{array}{lll}
\|u\|_{L_\infty(0,T;\dot B^s_{p,1}(\Omega))}
&\leq& \Sum_\ell\|U^\ell\|_{L_\infty(0,T;\dot B^s_{p,1}(\Omega))}\\[2ex]
&\lesssim& 
 \Sum_{0\leq\ell\leq M}\|U^\ell\|_{L_\infty(0,T;\dot B^s_{p,1}(\R^n))}
+\Sum_{M<\ell\leq L}\|V^\ell\|_{L_\infty(0,T;\dot B^s_{p,1}(\R^n_+))},\end{array}
$$
and similar inequalities for the other terms
of the l.h.s of \eqref{b5}. 
Of course, Proposition \ref{p:product} ensures that 
$$
\|u_0^\ell\|_{\dot B^s_{p,1}(\Omega)}\lesssim
\|u_0\|_{\dot B^s_{p,1}(\Omega)}\quad\hbox{and}\quad
 \|\tilde\eta^\ell f\|_{L_1(0,T;\dot B^s_{p,1}(\Omega))}\lesssim
 \|f\|_{L_1(0,T;\dot B^s_{p,1}(\Omega))}.
 $$
So using also \eqref{i5} and \eqref{i5bis} and assuming that $T$ is small enough, we end up with 
$$\displaylines{
\|u\|_{L_\infty(0,T;\dot B^s_{p,1}(\Omega))}
+\|(u_t,\nabla^2u)\|_{L_1(0,T;\dot B^s_{p,1}(\Omega))}
\lesssim 
\|u_0\|_{\dot B^s_{p,1}(\Omega)}\hfill\cr\hfill
 +\|f\|_{L_1(0,T;\dot B^s_{p,1}(\Omega))} 
+(\lambda^{-1} T^{1/2}+\lambda^{-2}T) 
 \|u\|_{L_1(0,T;\dot B^{s+2}_{p,1}(K))\cap L_\infty(0,T;\dot B^{s}_{p,1}(K))}.}
 $$
 Hence if in addition $\lambda^{-2}T$ is small enough,
$$
\displaylines{
\|u\|_{L_\infty(0,T;\dot B^s_{p,1}(\Omega))}+\|u_t,\nabla^2 u\|_{L_1(0,T;\dot B^s_{p,1}(\Omega))} 
\leq  C\Bigl(\|u_0\|_{B^s_{p,1}(\R^n)} 
+ \|f\|_{L_1(0,T;\dot B^s_{p,1}(\Omega))}\Bigr).}
$$
Repeating the argument over the interval $[T,2T]$ and so on, 
we get exactly Inequality \eqref{estdep}.
\medbreak

If we want to remove the time-dependency then it is just a matter of starting  from \eqref{b8} and \eqref{i8bis} instead
of \eqref{b5} and \eqref{i5}.
After a few computation and thanks to Young's inequality, we get for
some constant $C$ depending on $\lambda,$
$$\displaylines{
\|u_t,\nabla^2 u\|_{L_1(0,T;\dot B^s_{p,1}(\Omega))} \leq
 C(\|u_0\|_{\dot B^s_{p,1}(\Omega)}+\|f\|_{L_1(0,T;\dot B^s_{p,1}(\R^n_+))}
+\|u\|_{ L_1(0,T;\dot B^s_{p,1}(K))}).}
$$
For completeness, let us  say a few words about the existence, which is rather 
standard issue (see e.g. \cite{LaUrSol}).
If the domain is smooth then
the easiest approach is via the $L_2$-framework and Galerkin method. We may 
consider smooth approximations of data $f$ and $u_0$, such that to keep them in the space $H^m$
with sufficiently large $m\in \N$. Then the energy method provides us with 
approximate solutions in Sobolev spaces $H^m$ with large $m.$ 
In particular, the above a priori estimates  \eqref{estdep}  may be derived
for such solutions. It is then easy  to pass to the limit.

\end{proof}

\begin{remark} Let us emphasize that the term $\|u\|_{L_1(0,T;\dot B^s_{p,1}(K))}$ may be replaced by other
  lower order norms such as $\|u\|_{L_1(0,T;\dot B^{s'}_{p,1}(K))}$ with $s'\not=s$ close to $0.$
  In particular $s'$ may be  put to zero. 
\end{remark}

In order to complete the proof of Theorem \ref{th:heat-fin}, we now have to 
bound the last term of \eqref{estaux}, namely  $\|u\|_{L_1(0,T;\dot B^s_{p,1}(K))},$
 \emph{independently of} $T.$
This is the goal of the next lemma (where we keep the assumption that $\nu=1$). 
We here adapt to the heat equation  an approach that has been proposed
for the Stokes system in \cite{MarSol}.

\begin{lemma}\label{l:vitesse} Assume that $n\geq3$ and that  $1<p<n/2.$
Then for any 
$s\in(-1+1/p,1/p-2/n)$  
sufficiently smooth solutions to \eqref{eq:heat}  fulfill
$$
\|u\|_{L_1(0,T;\dot B^s_{p,1}(K))} \leq C(\|f\|_{L_1(0,T;\dot B^s_{p,1}(\Omega))}+ \|u_0\|_{\dot B^s_{p,1}(\Omega)}),
$$
where $C$ is \emph{independent of  $T$.}
\end{lemma}

\begin{proof}
Thanks to the linearity of the system, one may  split the solution $u$
into two parts, the first one $u_1$ being the solution of the system with zero initial data and source term $f$, and the second one $u_2,$
 the solution of the system with no source
term and initial data $u_0.$ In other words,  $u=u_1+u_2$ with $u_1$ and
$u_2$ satisfying 
\begin{equation}\label{lv14}
 \begin{array}{lcrlcr}
 u_{1,t}-\Delta u_1 =f& {\rm in} &(0,T)\times\Omega,\qquad\qquad & u_{2,t}-\Delta u_2 =0 & {\rm in} &(0,T)\times\Omega,\\
u_1=0 & {\rm on} & (0,T)\times\d\Omega,\qquad\qquad  & u_2=0 & {\rm on} &(0,T)\times \d \Omega,\\
u_1|_{t=0}=0& {\rm on} &\Omega,\qquad\qquad & u_2|_{t=0}=u_0 & {\rm on} &\Omega.
\end{array}
\end{equation}
Let us first focus on  $u_1$. Recall that up to a constant we have (see Proposition \ref{p:duality}): 
\begin{equation}\label{lv15}
\|u_1(t)\|_{\dot B^s_{p,1}(K)}=\sup \int_Ku_1(t,x) \eta_0(x)\, dx,
\end{equation}
where the supremum is taken over all  $\eta_0 \in \dot B^{-s}_{p',\infty}(K)$ such 
that $\| \eta_0 \|_{ \dot B^{-s}_{p',\infty}(K)}=1.$
Of course, by virtue of Remark \ref{r:besov}, any such function $\eta_0$  may be extended by $0$ over $\R^n,$
and its extension still has a norm of order $1.$ So we may assume that the supremum 
is taken over all  
\begin{equation}\label{eq:eta0}
\eta_0 \in \dot B^{-s}_{p',\infty}(\R^n)\quad\hbox{with }\ \|\eta_0\|_{\dot B^{-s}_{p',\infty}(\R^n)}=1
\quad\hbox{and}\quad\Supp\eta_0\subset K.
\end{equation}
Consider 
the solution $\eta$  to the following problem:
\begin{equation}\label{lv1}
\begin{array}{lcr}
\eta_t-\Delta \eta =0 \qquad & \hbox{in} & (0,T)\times\Omega,\\
\eta=0 & \hbox{on} & (0,T)\times\d \Omega,\\
\eta|_{t=0}=\eta_0 & \hbox{on} & \Omega.
\end{array}
\end{equation}
Testing the equation for  $u_1$ by $\eta(t-\cdot)$ we discover that
\begin{equation}\label{lv2}
\int_{\Omega} u_1(t,x) \eta_0(x)\,dx=\int_0^t \int_\Omega f(\tau,x) \eta(t-\tau,x) \,dx\,d\tau.
\end{equation}
The general theory for the heat  operator in exterior domains
implies the following estimates:
\begin{equation}\label{lv3}
\|\eta(t)\|_{L_a(\Omega)} \leq C\|\eta_0\|_{L_b (\Omega)} t^{-\frac{n}{2}(\frac{1}{b}-\frac{1}{a})}
\quad\hbox{for}\quad 1<b\leq a<\infty,
\end{equation}
 as well as
\begin{equation}\label{lv3bis}
\|\Delta \eta(t)\|_{L_a(\Omega)} \leq 
C\|\Delta\eta_0\|_{L_b (\Omega)} t^{-\frac{n}{2}(\frac{1}{b}-\frac{1}{a})}\quad\hbox{for}\quad 1<b\leq a<\infty.
\end{equation}

In the case $\Omega=\R^n,$
 those two inequalities may be derived easily from the (explicit) heat kernel. 
 To prove \eqref{lv3} in the case of an exterior domain,  it is enough to look at solutions to \eqref{lv1} as  subsolutions to the problem in  the whole space. 
 More precisely, if we assume that  $\eta_0\geq 0$ (this is not
 restrictive for one may consider the positive and negative part of the initial data separately), we get a solution to 
\eqref{lv1} defined over $(0,\infty)\times\Omega$ such that $\eta  \geq 0$. Then we consider an extension
$E\eta:\R^n \to \R$ of $\eta$, 
such that $E\eta = \eta$ for $x\in \Omega$ and $E\eta =0$ for $x\notin \Omega$. We claim that $E\eta$ is a subsolution to the Cauchy problem
\begin{equation}\label{cau}
 \bar \eta_t - \Delta \bar \eta =0 \mbox{ in } (0,T)\times\R^n \  \mbox{ with }\  \bar \eta|_{t=0}=E\eta_0.
\end{equation}
It is sufficient to show that $\eta \leq \bar \eta$, since $\bar \eta$ is always nonnegative. It is clear that 
\begin{equation}
 (\eta-\bar \eta)_t -\Delta (\eta- \bar\eta) =0 \ \mbox{ in } \ (0,T)\times\Omega.
\end{equation}
Consider $(\eta -\bar \eta)_+:=\max \{ \eta -\bar \eta, 0 \}.$ It    is obvious that $(\eta -\bar \eta)_+$ vanishes at the boundary, because $\eta$
is zero and $\bar \eta$ is nonnegative there. Hence we conclude
\begin{equation}
 \frac 12 \frac{d}{dt} \int_\Omega (\eta-\bar \eta)^2_+ dx + \int_\Omega |\nabla (\eta-\bar \eta)_+|^2 dx=0.
\end{equation}
Thus, $(\eta-\bar \eta)_+\equiv 0$, since $(\eta-\bar \eta)_+|_{t=0}=0$. So $\eta$ is bounded by $\bar \eta$.

 To prove \eqref{lv3bis} we  observe that for the smooth solutions the equation implies that 
$\Delta \eta|_{\d\Omega} =0$, so we can consider the problem on 
$\Delta \eta$  instead of $\eta$.
Now, as $\eta$ vanishes at the boundary, we have (see e.g. \cite{GT})
\begin{equation}\label{lv3ter}
\|\nabla^2 \eta\|_{L_c(\Omega)}\leq  \| \Delta \eta\|_{L_c(\Omega)}\quad\hbox{for all }\ 1<c<\infty.
\end{equation}
Hence, interpolating between \eqref{lv3} and \eqref{lv3bis} yields
 for $0<s<1/b.$

\begin{equation}\label{lv6}
\|\eta(t) \|_{\dot B^s_{b,r}(\Omega)} \leq C \|\eta_0\|_{\dot B^s_{a,r}(\Omega)} t^{-\frac{n}{2}(\frac{1}{p}-\frac{1}{q})}
\mbox{ ~~ for ~ } 1< a\leq b <\infty\ \hbox{ and }\ 1\leq r\leq\infty.
\end{equation}
In order to extend this inequality  to negative indices $s,$ 
we consider the following dual problem:
\begin{equation}\label{lv7}
\begin{array}{lcc}
\zeta_t-\Delta \zeta =0 & \hbox{in} & (0,T)\times\Omega,\\
\zeta=0 & \hbox{on} & (0,T)\times\d \Omega,\\
\zeta|_{t=0}=\zeta_0 & \hbox{on} & \Omega,
\end{array}
\end{equation}
where  $\zeta_0 \in B^{-s}_{b',r'}(\Omega)$. 

Now, testing \eqref{lv7} by $\eta(t-\cdot)$ yields
\begin{equation}\label{lv8}
\int_\Omega \eta(t,x) \zeta_0(x) \,dx =\int_\Omega \eta_0(x) \zeta(t,x)\, dx.
\end{equation}
Let us  observe that
\begin{equation}\label{lv9}
\|\eta(t)\|_{\dot B^s_{b,r}(\Omega)} =\sup_{\zeta_0} \int_{\Omega}
\eta(t,x) \zeta_0(x)\, dx,
\end{equation}
where the supremum is taken over all $\zeta_0 \in \dot B^{-s}_{b',r'}(\Omega)$ such that 
$\|\zeta_0\|_{ \dot B^{-s}_{b',r'}(\Omega)}=1$. Thus by virtue of \eqref{lv8}, we get:
\begin{equation}\label{lv10}
\|\eta(t)\|_{\dot B^s_{b,r}(\Omega)}=\sup_{\zeta_0} \int_\Omega \eta_0(x) \zeta(t,x)\,dx 
\leq \sup_{\zeta_0}
\Bigl(\|\eta_0\|_{\dot B^s_{a,r}(\Omega)} \|\zeta(t)\|_{\dot B^{-s}_{a',r'}(\Omega)}\Bigr). 
\end{equation}
Since $-s$ is positive we can apply \eqref{lv6}  and get if $0<-s<1/a',$
$$
\|\eta(t)\|_{\dot B^s_{b,r}(\Omega)} 
\leq C\|\eta_0\|_{\dot B^s_{a,r}(\Omega)} t^{-\frac n2 (\frac{1}{a'}-\frac{1}{b'})}
\sup_{\zeta_0} \|\zeta_0\|_{\dot B^{-s}_{b',r'}(\Omega)}. 
$$
Since $\frac{1}{a'}-\frac{1}{b'}=\frac 1b -\frac 1a,$ we conclude that
\begin{equation}\label{lv12}
\|\eta(t)\|_{\dot B^s_{a,r}(\Omega)}\leq 
C\|\eta_0\|_{\dot B^s_{b,r}(\Omega)}  t^{-\frac n2 (\frac{1}{b}-\frac{1}{a})}\quad\hbox{if }\ s>-1+1/a. 
\end{equation}

In order to get the remaining case $s=0,$ 
it suffices to argue by interpolation between \eqref{lv6} and \eqref{lv12}. 
One can thus conclude that 
for all  $1<b\leq a<\infty,$ $q\in [1,\infty]$  and $-1+1/a<s<1/b,$  we have 
\begin{equation}\label{lv13}
\|\eta(t)\|_{\dot B^s_{a,r}(\Omega)}\leq C\|\eta_0\|_{\dot B^s_{b,r}(\Omega)}  t^{-\frac n2 (\frac{1}{b}-\frac{1}{a})}.  
\end{equation}

Now we return to the initial problem of bounding $u_1.$ 
Starting from \eqref{lv2} and using duality, one may write
$$
\biggl|\int_\Omega u_1(t,x) \eta_0(x)\,dx\biggr|\lesssim\int_0^t \|f(\tau)\|_{\dot B^s_{p,1}(\Omega)} 
\|\eta(t-\tau)\|_{\dot B^{-s}_{p',\infty}(\Omega)}\,d\tau. 
$$
Hence   splitting  the interval $(0,t)$ into 
$(0,\max(0,t-1))$ and $(\max(0,t-1),t)$ and applying \eqref{lv13} yields for any $\ep\in(0,1+s),$
$$
\displaylines{
\biggl|\int_\Omega u_1(t,x) \eta_0(x)\,dx\biggr|\lesssim\int_{\max(0,t-1)}^t \|f(\tau)\|_{\dot B^s_{p,1}(\Omega)} 
\|\eta_0\|_{\dot B^{-s}_{p',\infty}(\Omega)}\,d\tau 
\hfill\cr\hfill+
\int_0^{\max(0,t-1)} \|f(\tau)\|_{\dot B^s_{p,1}(\Omega)} \|\eta_0\|_{\dot B^{-s}_{\frac{1}{1-\ep},\infty}(\Omega)}
(t-\tau)^{-\frac n2(\frac 1p-\ep)}\,d\tau.}
$$
Now, as $\eta_0$ is supported in $K,$ one has
$\| \eta_0 \|_{\dot B^{-s}_{a,\infty}(\Omega)} \leq C|K|^{\frac{1}{p}+\frac1a-1}
\|\eta_0 \|_{\dot B^{-s}_{p',\infty}(\Omega)}.$
This may easily proved by introducing a suitable smooth  cut-off function with value
$1$ over $K$ and taking advantage of Proposition \ref{p:product}. 
A scaling argument yields the dependency of the norm of the embedding 
with respect to $|K|.$ 
Hence we have for some constant $C$ depending on~$K$:
$$ \|\eta_0\|_{\dot B^{-s}_{\frac{1}{1-\ep},\infty}}\leq C \|\eta_0 \|_{\dot B^{-s}_{p',\infty}(\Omega)}.$$
So, keeping in mind \eqref{lv2} and the fact that the supremum 
is taken over all the functions $\eta_0$ satisfying \eqref{eq:eta0}, 
we deduce that
$$\displaylines{\quad
\| u_1(t)\|_{\dot B^s_{p,1}(K)} \leq C \biggl( \int_{\max(0,t-1)}^t \|f(\tau)\|_{\dot B^s_{p,1}(\Omega)}\, d\tau
\hfill\cr\hfill
+\int_0^{\max(0,t-1)}(t-\tau)^{-\frac n2 (\frac{1}{p}-\epsilon)} \|f(\tau)\|_{\dot B^s_{p,1}(\Omega)}\, d\tau\biggr).\quad}
$$
Therefore,
\begin{equation}\label{lv19}
\int_1^T \| u_1\|_{\dot B^s_{p,1}(K)}\, dt \leq C\biggl(1+ \int_1^T \tau^{-\frac n2 (\frac{1}{p}-\epsilon)}\, d\tau\biggr)
\int_0^T \|f\|_{\dot B^s_{p,1}(\Omega)} \,dt.
\end{equation}
For the time interval $[0,1],$ we merely have  
\begin{equation*}\label{lv20}
\int_0^1 \|  u_1\|_{\dot B^s_{p,1}(K)} \,dt \leq C
\int_0^1\|f\|_{\dot B^s_{p,1}(\Omega)}\, dt.
\end{equation*}

Now, 
 provided that one may find some $\ep>0$ such  that 
\begin{equation}\label{lv18}
\frac n2 \Bigl(\frac{1}{p}-\epsilon\Bigr)>1,
\end{equation}
a condition which is equivalent to $p<n/2,$  the constant in 
\eqref{lv19} may be made independent of $T.$ Hence
 we conclude that
\begin{equation}
 \int_0^T \|u_1\|_{\dot B^s_{p,1}(K)}\,dt \leq C\int_0^T
 \|f\|_{\dot B^s_{p,1}(\Omega)}\,dt
\end{equation}
with $C$ independent of  $T$.
\medbreak
Let us now bound $u_2.$ We first write that
\begin{equation}\label{lv21}
\|u_2(t)\|_{\dot B^s_{p,1}(K)}  \leq C \|u_0\|_{\dot B^s_{p,1}(\Omega)} 
\end{equation}
and, if $-1+\ep<s<1/p,$
\begin{equation*}\label{lv22}
\|u_2(t)\|_{\dot B^s_{p,1}(K)}  \leq C|K|^{\frac{1}{p}-\epsilon} \|u_2(t)\|_{\dot B^s_{\frac 1\epsilon,1}(K)}\leq 
 C|K|^{\frac{1}{p}-\epsilon}  \|u_0\|_{\dot B^s_{p,1}(\Omega)} t^{-\frac n2 (\frac 1p -\epsilon)}.
\end{equation*}
Then decomposing the integral over $[0,T]$
into an integral over $[0,\min(1,T)]$ and  $[\min(1,T),T],$ we easily get
\begin{equation}\label{lv23}
\int_0^T \|u_2(t)\|_{\dot B^s_{p,1}(K)} \, dt \leq C\biggl(1+\int_{\min(1,T)}^T t^{-\frac n2 (\frac 1p -\epsilon)} \,dt\biggr)
\| u_0\|_{\dot B^s_{p,1}(\Omega)}. 
\end{equation}
The integrant in the r.h.s. of \eqref{lv23} is finite whenever  \eqref{lv18} holds.
Hence, \begin{equation}\label{lv24}
\int_0^T \|u_2(t)\|_{\dot B^s_{p,1}(K)}\,  dt \leq C\|u_0\|_{\dot B^s_{p,1}(\Omega)}.
\end{equation}

Putting this together with \eqref{lv19} and \eqref{lv20} completes
the proof of the lemma. 
\end{proof}
\medbreak
We are now ready to prove Theorem \ref{th:heat-fin}.
Granted with Theorem \ref{th:ext}, it is enough to show that 
$\|u\|_{L_1(0,T;\dot B^s_{p,1}(K)\cap \dot B^{s'}_{q,1}(K))}$ may be bounded by 
the right-hand side of \eqref{fin1}. 

As a matter of fact $\|u\|_{L_1(0,T;\dot B^{s'}_{q,1}(K))}$ may be directly bounded 
from Lemma \ref{l:vitesse}, and the same holds for 
$\|u\|_{L_1(0,T;\dot B^{s}_{p,1}(K))}$ if $p<n/2.$ 

If  $p\geq n/2,$ then we use the fact  so that
$$
\dot B^{s+2}_{q,1}(\Omega) \subset \dot B^s_{q^*,1}(\Omega)
\mbox{ ~~ with ~~ } \frac{1}{q^*}=\frac 1q - \frac 2n\cdotp
$$
Therefore, if $q<n/2\leq p < q^*$ then   one may combine  interpolation 
and Lemma \ref{l:vitesse} so as to absorb   
$\|u\|_{L_1(0,T;\dot B^s_{p,1}(K))}$ by the left-hand side of \eqref{fin1},
changing the constant $C$ if necessary. 

If $p\geq q^*$ then one may repeat the argument again and again
until  the all possible values of $p$ in $(n/2,\infty)$ are exhausted. 
 Theorem \ref{th:heat-fin} is proved.


\subsection{The bounded domain case}

We end this section with a few remarks concerning the case where $\Omega$ is a bounded
domain of $\R^n$ with $n\geq2.$ 
Then the proof of Theorem \ref{th:ext} is similar : we still have to introduce some suitable 
resolution of unity $(\eta^\ell)_{0\leq \ell\leq L}.$ The only difference is that,  now, 
$\eta^0$ has compact support. Hence  
 Theorem \ref{th:ext}   holds true with  $K=\overline\Omega.$
\medbreak
In order to remove the time dependency in the estimates, 
we use the fact (see e.g. \cite{Fridman})  that the solution $\eta$ to 
\eqref{lv1} satisfies for some $c>0,$
$$
\|\eta(t)\|_{L_p(\Omega)}\leq Ce^{-ct}\|\eta_0\|_{L_p(\Omega)},
$$
which also implies that
$$
\|\nabla^2\eta(t)\|_{L_p(\Omega)}\leq Ce^{-ct}\|\nabla^2\eta_0\|_{L_p(\Omega)}.
$$
Hence we  have for any $1<p<\infty$ and $-1+1/p<s<1/p,$
\begin{equation}\label{lv25}
\|\eta(t)\|_{\dot B^s_{p,1}(\Omega)}\leq Ce^{-ct}\|\eta_0\|_{\dot B^s_{p,1}(\Omega)}.
\end{equation}
Defining $u_1$ and $u_2$ as in \eqref{lv14}, one may thus write
$$
\biggl|\int_\Omega u_1(t,x)\eta_0(x)\,dx\biggr|\lesssim\|\eta_0\|_{\dot B^{-s}_{p',\infty}(\Omega)}
\int_0^t\|f(\tau)\|_{\dot B^s_{p,1}(\Omega)}e^{-c(t-\tau)}\,d\tau,
$$
thus giving 
$$
\|u_1\|_{L_1(0,T;\dot B^s_{p,1}(K))}\lesssim\|f\|_{L_1(0,T;\dot B^s_{p,1}(\Omega))}.
$$
Of course, we also have
$$
\|u_2\|_{L_1(0,T;\dot B^s_{p,1}(K))}\lesssim\|u_0\|_{\dot B^s_{p,1}(\Omega)}.
$$
So one may conclude that Lemma \ref{l:vitesse} holds true for \emph{any} $1<p<\infty$
and $-1+1/p<s<1/p.$ Consequently, we get:
\begin{theorem}\label{th:bounded} If $1<p<\infty$ and $-1+1/p<s<1/p$ then
the statement of Theorem \ref{th:half} remains true in any 
$C^2$ bounded domain.
\end{theorem}


\section{Applications}

In this last section, we give some application of the maximal regularity estimates that
have been proved hitherto. 
As an example, we prove global stability results (in a critical functional framework) for 
trivial/constant solutions to the following system:
\begin{equation}\label{ap1}
 \begin{array}{lcr}
  \vec u_t - \nu\Delta \vec u +P\cdot \nabla^2 \vec u= f_0(\vec u)+f_1(\vec u)\cdot\nabla \vec u & \mbox{ in } & (0,T)\times\Omega, \\
\vec u=0 & \mbox{ at } & (0,T)\times\d \Omega, \\
\vec u|_{t=0} = u_0 & \mbox{ on } & \Omega.
 \end{array}
\end{equation}
Above, $\nu$ is a positive parameter, $\vec u$ stands for a $r$-dimensional vector and 
$P=(P_1,\cdots,P_r)$  where the $P_k$'s are $n\times n$  matrices with 
suitably smooth coefficients.
The nonlinearities $f_0:\R^r\rightarrow\R^r$ and $f_1:\R^r\rightarrow\cM_{r,n}(\R)$
are $C^1$  and satisfy 
\begin{equation}\label{eq:nonlinear}
  f_0(0)=0, \quad df_0(0)=0\  \mbox{ and }\  f_1(0)=0,
\end{equation}
together with some growths conditions that will be detailed below. 
\medbreak
As we have in mind applications to Theorem \ref{th:heat-fin}, we focus on the case where
 $\Omega$ is a smooth exterior domain of $\R^n$ with $n\geq3.$
Of course, based on our other maximal regularity results, similar (and somewhat easier)
statements may be proved for bounded domains, $\R^n_+$  or $\R^n.$ 
\medbreak
Here are two  important examples entering in the class of equations \eqref{ap1}.
The first one is {\it the nonlinear heat transfer equation} (see \cite{Xin} and the references therein):
\begin{equation}\label{ap2}
 u_t-\nu\Delta u=f(u).
\end{equation}
A  classical form of the nonlinearity is $f(u)=Ku^2(u-u^*).$
However one may consider  more complex models describing a flame propagation like in \cite{LM}.
\medbreak
The second example is the  {\it viscous Burgers equation} \cite{Hill,Hor}
\begin{equation}\label{ap3}
 u_t+u \d_{x_1} u -\nu \Delta u =0.
\end{equation}
which   enters in the  class of models like
\begin{equation}\label{ap4}
 \vec u_t-\nu\Delta \vec u =B(\vec u,\nabla \vec u).
\end{equation}
In the case where $B(\vec u,\nabla\vec u)=-\vec u\cdot\nabla\vec u,$
this is just the equation for pressureless viscous gases with constant density.

Below, based on Theorem \ref{th:heat-fin}, we shall prove two global-in-time
results concerning the stability of the trivial solution of  System \eqref{ap1}.  
In the first statement, to  simplify the presentation, we only consider the case where the data 
belong to spaces with regularity index equals to $0.$ To simplify the notation, we omit
the dependency with respect to the domain $\Omega$ in all that follows.
\begin{theorem}\label{th:non}
Let $1<q<n/2$ and $\Omega$ be an exterior domain of $\R^n$ ($n\geq3$). 
There exist two positive constants $\eta$ and $c_\nu$  such that
for all $P:[0,\infty)\times\Omega\rightarrow\R^n\times\R^n\times\R^k$
satisfying\footnote{Below $\cM(X)$ denotes the \emph{multiplier space} associated
to the Banach space $X,$ that is the set of those functions $f$ such that 
$fg\in X$ whenever $g$ is in $X$ endowed with the norm
$\|f\|_{\cM(X)}:=\inf_g \|fg\|_X$ where the infimum is taken over all $g\in X$ with norm $1.$}:
\begin{equation}\label{eq:smallP} 
\|P\|_{L_\infty(0,\infty;\cM(\dot B^0_{n,1}\cap\dot B^0_{q,1}))} \leq\eta \nu,
\end{equation}
for all nonlinearities $f_0$ and $f_1$ fulfilling \eqref{eq:nonlinear} and
\begin{equation}\label{eq:growth1}
|df_0(\vec w)| \leq C|\vec w|, \qquad |df_1(\vec w)|\leq C,
\end{equation}
and for all $\vec u_0 \in \dot B^0_{n,1}\cap \dot  B^0_{q,1}$ such that
\begin{equation}\label{ap5}
 \|\vec u_0\|_{\dot B^0_{n,1}\cap \dot B^0_{q,1}} \leq c_\nu,
\end{equation}
System \eqref{ap1} admits a unique global solution $\vec u$ in the space
\begin{equation}\label{ap6}
\cC_b(0,\infty;\dot B^0_{n,1}\cap \dot B^0_{q,1})
\cap L_1(0,\infty; \dot B^2_{n,1}\cap \dot B^2_{q,1}).
\end{equation}
\end{theorem}
\begin{proof} 
Granted with Theorem \ref{th:heat-fin}, 
 the result mainly relies on embedding, composition and  and product estimates in Besov spaces.
We focus on the proof of a priori estimates for a global solution $\vec u$ to  \eqref{ap1}.
 First,  applying Theorem \ref{th:heat-fin} yields
\begin{multline}\label{ap6a}
 \|\vec u\|_{L_\infty(0,\infty;\dot B^0_{n,1}\cap \dot B^0_{q,1})
\cap L_1(0,\infty;\dot B^2_{n,1}\cap \dot B^2_{q,1})} 
\lesssim \|P\|_{L_\infty(0,\infty;\cM(\dot B^0_{n,1}\cap\dot B^0_{q,1}))} \|\vec u \|_{L_1(0,\infty;\dot B^2_{n,1}\cap \dot B^2_{q,1})}\\
+\|\vec u_0\|_{\dot B^0_{n,1}\cap \dot B^0_{q,1}}+\|f_0(\vec u)\|_{L_1(0,\infty;\dot B^0_{n,1}\cap \dot B^0_{q,1})}
+\|f_1(\vec u)\cdot \nabla \vec u\|_{L_1(0,\infty;\dot B^0_{n,1}\cap \dot B^0_{q,1})}.
\end{multline}
Bounding  the last two terms follows from Propositions \ref{p:product} and  \ref{p:compo}. More precisely, for $p=q,n,$ we have
$$\begin{array}{lll}
\|f_1(\vec u)\cdot\nabla \vec u\|_{\dot B^0_{p,1}}&\lesssim&
 \|f_1(\vec u)\|_{\dot B^{1/2}_{n,1}}\|\nabla\vec u\|_{\dot B^{1/2}_{p,1}}\\
  &\lesssim&
 \|\vec u\|_{\dot B^{1/2}_{n,1}}\|\nabla\vec u\|_{\dot B^{1/2}_{p,1}}.\end{array}
  $$
  Therefore, applying H\"older inequality,
  $$\|f_1(\vec u)\cdot\nabla \vec u\|_{L_1(0,\infty;\dot B^0_{p,1})}
  \lesssim 
    \|\vec u\|_{L_4(0,\infty;\dot B^{1/2}_{n,1})}\|\nabla\vec u\|_{L_{4/3}(0,\infty;\dot B^{1/2}_{p,1})},
    $$
    whence, using elementary interpolation,
\begin{equation}\label{ap17}
 \|f_1(\vec u)\cdot\nabla\vec u\|_{L_1(0,\infty;\dot B^0_{n,1}\cap \dot B^0_{q,1})} \lesssim
 \|\vec u\|^2_{L_\infty(0,\infty;\dot B^0_{n,1}\cap \dot B^0_{q,1})
\cap L_1(0,\infty;\dot B^2_{n,1}\cap \dot B^2_{q,1})}.
\end{equation}
Bounding $f_0(\vec u)$ is slightly more involved. To handle the norm in $L_1(0,\infty;\dot B^0_{n,1}(\Omega)),$
we use the following critical embedding:
$$\dot B^1_{n/2,1}\hookrightarrow \dot B^{0^+}_{n^-,1}\hookrightarrow \dot B^0_{n,1}.$$
 Hence Proposition \ref{p:product} enables us to write that
$$
\begin{array}{lll}
\|f_0(\vec u)\|_{\dot B^0_{n,1}}&\lesssim&\|f_0(\vec u)\|_{\dot B^{0^+}_{n^-,1}},\\
&\lesssim& \|\vec u\|_{L_\infty}\|\vec u\|_{\dot B^{0^+}_{n^-,1}},\\
&\lesssim& \|\vec u\|_{L_\infty}\|\vec u\|_{\dot B^{1}_{n/2,1}},\\
&\lesssim& \|\vec u\|_{\dot B^1_{n,1}}\|\vec u\|_{\dot B^{1}_{q,1}\cap \dot B^1_{n,1}}.
\end{array}
$$
The last inequality stems from the embedding $\dot B^1_{n,1}\hookrightarrow L_\infty$
and from the fact that $q<n/2<n,$ whence
$$
\dot B^{1}_{q,1}\cap \dot B^1_{n,1}\hookrightarrow \dot B^1_{n/2,1}.
$$
Therefore, using H\"older inequality and elementary interpolation, we deduce that 
\begin{equation}\label{ap10}
 \|f_0(\vec u) \|_{L_1(0,\infty;\dot B^0_{n,1})} \lesssim \|\vec u\|^2_{L_\infty(0,\infty;\dot B^0_{n,1}\cap \dot B^0_{q,1})
\cap L_1(0,\infty;\dot B^2_{n,1}\cap \dot B^2_{q,1})}.
\end{equation}
Finally  we have to bound $f_0(\vec u)$ in $L_1(0,\infty;\dot B^0_{q,1}).$ 
For that it suffices to estimate it in $L_1(0,\infty;\dot B^{0^+}_{q,1})$ and in $L_1(0,\infty;L_{q^-}).$
Indeed we observe that $L_{q^-}\hookrightarrow \dot B^{0^-}_{q,\infty},$ and thus
\begin{equation}\label{ap11}
L_1(0,\infty;\dot B^{0^+}_{q,1})\cap L_1(0,\infty;L_{q^-})\hookrightarrow
L_1(0,\infty;\dot B^0_{q,1}).
\end{equation}
Now, on the one hand, according to  Proposition \ref{p:compo} and H\"older inequality we have
$$
\|f_0(\vec u)\|_{L_1(0,\infty;\dot B^{0^+}_{q,1})}\lesssim 
\|\vec u\|_{L_{1^+}(0,\infty;L_{\infty})}\|\vec u\|_{L_{\infty^-}(0,\infty;\dot B^{0^+}_{q,1})}.
$$
By interpolation, we easily get 
$$
\|\vec u\|_{L_{\infty^-}(0,\infty;\dot B^{0^+}_{q,1})}\lesssim
 \|\vec u\|_{L_{\infty}(0,\infty;\dot B^{0}_{q,1})\cap L_1(0,\infty;\dot B^2_{q,1})}
$$
and because $q<n/2,$ 
$$
\begin{array}{lll}
\|\vec u\|_{L_{1^+}(0,\infty;L_{\infty})}&\lesssim&\|\vec u\|_{L_1(0,\infty;\dot B^0_{\infty,1})
\cap L_2(0,\infty;\dot B^0_{\infty,1})},\\
&\lesssim&\|\vec u\|_{L_1(0,\infty;\dot B^2_{n/2,1})\cap L_2(0,\infty;\dot B^1_{n,1})},\\
&\lesssim&\|\vec u\|_{L_1(0,\infty;\dot B^2_{q,1}\cap \dot B^2_{n,1})\cap L_1(0,\infty;\dot B^2_{n,1})
\cap L_\infty(0,\infty;\dot B^0_{n,1})}.
\end{array}
$$
Therefore we have, 
\begin{equation}\label{ap12}
\|f_0(\vec u)\|_{L_1(0,\infty;\dot B^{0^+}_{q,1})}\lesssim 
\|\vec u\|_{L_1(0,\infty;\dot B^2_{q,1}\cap \dot B^2_{n,1})
\cap L_\infty(0,\infty;\dot B^0_{q,1}\cap\dot B^0_{n,1})}^2.
\end{equation}
On the other hand, using the fact that $|f_0(\vec u)|\leq C|\vec u|^2$ and H\"older inequality, we may write
$$
\|f_0(\vec u)\|_{L_1(0,\infty;L_{q^-})}\leq\|\vec u\|_{L_\infty(0,\infty;L_q)}\|\vec u\|_{L_1(0,\infty;L_{\infty^-})}.
$$
We obviously have $\dot B^0_{q,1}\hookrightarrow L_q$ and, because $2-n/q<0,$
$$
\dot B^2_{q,1}\cap \dot B^2_{n,1}\hookrightarrow L_{\infty^-}.
$$
Therefore
\begin{equation}\label{ap13}
\|f_0(\vec u)\|_{L_1(0,\infty;L_{q^-})}\lesssim \|\vec u\|_{L_\infty(0,\infty;\dot B^0_{q,1})}
\|\vec u\|_{L_1(\dot B^2_{q,1}\cap \dot B^2_{n,1})}.
\end{equation}
So putting \eqref{ap12} and \eqref{ap13} together and taking advantage of \eqref{ap11}, 
we end up with 
\begin{equation}\label{ap14}
 \|f_0(\vec u) \|_{L_1(0,\infty;\dot B^0_{q,1})} \lesssim \|\vec u\|^2_{L_\infty(0,\infty;\dot B^0_{n,1}\cap \dot B^0_{q,1})
\cap L_1(0,\infty;\dot B^2_{n,1}\cap \dot B^2_{q,1})}.
\end{equation}
It is now time to plug \eqref{ap17}, \eqref{ap10} and \eqref{ap14} in 
\eqref{ap6a}. We get
\begin{multline}\label{ap18}
 \|\vec u\|_{L_\infty(0,\infty;\dot B^0_{n,1}\cap \dot B^0_{q,1})
\cap L_1(0,\infty;\dot B^2_{n,1}\cap \dot B^2_{q,1})} \\
\leq C( \|P\|_{L_\infty(0,\infty;\cM(\dot B^0_{n,1}\cap\dot B^0_{q,1}))} \|\vec u \|_{L_1(0,\infty;\dot B^2_{n,1}\cap \dot B^2_{q,1})}\\
+\|\vec u\|^2_{L_\infty(0,\infty;\dot B^0_{n,1}\cap \dot B^0_{q,1})
\cap L_1(0,\infty;\dot B^2_{n,1}\cap \dot B^2_{q,1})}+ \|\vec u_0\|_{\dot B^0_{n,1}\cap \dot B^0_{q,1}}).
\end{multline}

Obviously, the above estimate enables us to get a \emph{global-in-time} control of the solution in the desired functional space
whenever  \eqref{eq:smallP} and \eqref{ap5} are satisfied.
Starting from this observation and using the existence part of  Theorem \ref{th:heat-fin}, 
it is easy to prove Theorem \ref{th:non} by means of Banach fixed point theorem as in \cite{DM-jfa} for instance.
The details are left to the reader.
\end{proof}

\begin{theorem}\label{th:nonbis} Assume that $P\equiv0$ and that $f_1\equiv0.$
 Suppose that  $f_0$ satisfies \eqref{eq:nonlinear} and 
$$
|df_0(\vec w)|\leq C(|\vec w|^{m-1} + |\vec w|) \mbox{ for some } m \geq 2.
$$
Let  $1<q<\frac n2$ and   $q\leq p<\infty.$ Assume that 
$$
s_p:=\frac np-\frac2{m-1}\in\biggl(0,\frac1p\biggr)\quad\hbox{and}\quad 0<s_q<\frac1q-\frac2n\cdotp
$$
Then  there exists a constant $c_\nu$ such that  if
\begin{equation}\label{ap5a}
 \|\vec u_0\|_{\dot B^{s_p}_{p,1}\cap \dot B^{s_q}_{q,1}} \leq c_\nu
\end{equation}
then System \eqref{ap1} admits a unique  global-in-time  solution $\vec u$ such that
\begin{equation}\label{ap6aa}
\vec  u \in \cC_b(0,\infty;\dot B^{s_p}_{p,1}\cap \dot B^{s_q}_{q,1})
\cap L_1(0,\infty; \dot B^{2+s_p}_{p,1}\cap \dot B^{2+s_q}_{q,1}).
\end{equation}
\end{theorem}

\begin{proof}
Once again, we start from  Theorem \ref{th:heat-fin} which implies the following inequality:
\begin{multline}\label{x1}
 \|\vec u\|_{L_\infty(0,\infty;\dot B^{s_p}_{p,1}\cap \dot B^{s_q}_{q,1})\cap 
L_1(0,\infty;\dot B^{2+s_p}_{p,1}\cap \dot B^{2+s_q}_{q,1})} 
\lesssim\|f_0(\vec u)\|_{L_1(0,\infty;\dot B^{s_p}_{p,1}\cap \dot B^{s_q}_{q,1})}
+\|\vec u_0\|_{\dot B^{s_p}_{p,1}\cap \dot B^{s_q}_{q,1}}.
\end{multline}
Now (a slight generalization of) Proposition \ref{p:compo} ensures that for $s=s_p,s_q$ and for $r=p,q,$
$$
\|f_0(\vec u)\|_{\dot B^{s}_{r,1}}\lesssim \bigl(\|\vec u\|_{L_\infty}+\|\vec u\|_{L_{\infty}}^{m-1}\bigr)
\|\vec u\|_{\dot B^{s}_{r,1}}.
$$
Therefore,
\begin{multline}\label{x2}
 \|f_0(\vec u)\|_{L_1(0,\infty;\dot B^{s_p}_{p,1}\cap \dot B^{s_q}_{q,1}} \leq 
C\bigl( \| \vec u \|_{L_1(0,\infty;L_\infty)} \\+
\| \vec u \|_{L_{m-1}(0,\infty;L_\infty)}^{m-1}\bigr)\|\vec u\|_{L_\infty(0,\infty;\dot B^{s_p}_{p,1}\cap \dot B^{s_q}_{q,1})}.
\end{multline}
Hence it is only a matter of proving that the norm of $\vec u$ in $L_1(0,\infty;L_\infty)$
and in $L_{m-1}(0,\infty;L_\infty)$ may be bounded by means 
of the norm in $L_1(0,\infty;\dot B^2_{q,1}\cap \dot B^2_{n,1})\cap L_\infty(0,\infty;\dot B^0_{q,1}\cap\dot B^0_{n,1}).$
Now, we notice that
$\dot B^{s_p+2/(m-1)}_{p,1}$ embeds continuously in $L_\infty$ and that, by interpolation,
$$
\|\vec u\|_{L_{m-1}(0,\infty;\dot B^{s_p+2/(m-1)}_{p,1})}\leq
\|\vec u\|_{L_\infty(0,\infty;\dot B^{s_p}_{p,1})\cap L_1(0,\infty;\dot B^{s_p+2}_{p,1})}.
$$
Hence we do have 
\begin{equation}\label{x3}
\| \vec u \|_{L_{m-1}(0,\infty;L_\infty)}\lesssim \|\vec u\|_{L_\infty(0,\infty;\dot B^{s_p}_{p,1})\cap L_1(0,\infty;\dot B^{s_p+2}_{p,1})}.
\end{equation}
Finally, we notice that 
$\dot B^2_{s_q+2,1}\hookrightarrow \dot B^{2+s_q-n/q}_{\infty,1}$
and that $2+s_q-n/q<0$. At the same time $\dot B^{s_p+2}_{p,1}\hookrightarrow \dot B^1_{\infty,1},$
therefore 
$$
\dot B^{s_q+2}_{q,1}\cap \dot B^{s_p+2}_{p,1}\hookrightarrow L_\infty.
$$
Hence we have 
\begin{equation}\label{x4}
\| \vec u \|_{L_{1}(0,\infty;L_\infty)}\lesssim \|\vec u\|_{L_1(0,\infty;\dot B^{s_p+2}_{p,1})\cap L_1(0,\infty;\dot B^{s_q+2}_{q,1})}.
\end{equation}

Putting \eqref{x3} and \eqref{x4} into \eqref{x2}  and then into \eqref{x1} we get
$$\displaylines{
 \|\vec u\|_{L_\infty(0,\infty;\dot B^{s_p}_{p,1}\cap \dot B^{s_q}_{q,1})\cap 
L_1(0,\infty;\dot B^{2+s_p}_{p,1}\cap \dot B^{2+s_q}_{q,1})} 
\lesssim \|\vec u_0\|_{\dot B^{s_p}_{p,1}\cap \dot B^{s_q}_{q,1}}\hfill\cr\hfill
+(\|\vec u\|_{L_\infty(0,\infty;\dot B^{s_p}_{p,1}\cap \dot B^{s_q}_{q,1})\cap 
L_1(0,\infty;\dot B^{2+s_p}_{p,1}\cap \dot B^{2+s_q}_{q,1})} ^{m-2}+1)
\|\vec u\|_{L_\infty(0,\infty;\dot B^{s_p}_{p,1}\cap \dot B^{s_q}_{q,1})\cap 
L_1(0,\infty;\dot B^{2+s_p}_{p,1}\cap \dot B^{2+s_q}_{q,1})} ^2.}
$$
The smallness of the initial data in \eqref{ap6a}  enables to close the estimate for 
the left-hand side of  the above inequality. 
The existence issue is just a consequence of Banach fixed point theorem. This completes the proof of the theorem.
\end{proof}

\begin{remark} Even though System \eqref{ap1} does not have any scaling invariance in general, 
our two statements are somewhat critical from the regularity point of view. 
Indeed,  in the functional framework used in Theorem \ref{th:non} and under  the growth 
condition \eqref{eq:growth1}, the nonlinearity $f_0(\vec u)$ is lower order
compared to $f_1(\vec u)\cdot\nabla \vec u.$ Now, we notice that if $f_0\equiv0$ and $P\equiv0$
then the initial value problem for System \eqref{ap1} 
(in the $\R^n$ case) is invariant for all $\lambda>0$ under the transform:
$$
(u(t,x),u_0(x))\longrightarrow \lambda(u(\lambda^2t,\lambda x),u_0(\lambda x)).
$$
At the same time, the norm $\|\cdot\|_{\dot B^0_{n,1}(\R^n)}$
is invariant by the above rescaling for $u_0.$
\medbreak
As regards Theorem \ref{th:nonbis}, the nonlinearity $f_0(\vec w)$ is at most of order $m.$
Now, if (the coefficients of) $f_0(\vec w)$ are homogeneous polynomials of degree $m$ then 
the system is invariant by  
$$
(u(t,x),u_0(x))\longrightarrow \lambda^{\frac2{m-1}}(u(\lambda^2t,\lambda x),u_0(\lambda x)).
$$
Hence the regularity  $\dot B^{s_p}_{p,1}$ is critical. 
\end{remark}

\begin{acknowledgement}
 The second author has been supported by the MN grant IdP2011 000661.
\end{acknowledgement}

\end{document}